\documentclass[twoside]{article}
\usepackage[preprint]{aistats2027}
\usepackage[T1]{fontenc}
\usepackage{lmodern}
\usepackage{amsmath,amssymb,amsthm,mathtools}
\usepackage{graphicx,booktabs,multirow,array}
\usepackage[round]{natbib}
\usepackage{microtype}
\usepackage{xcolor}
\usepackage{float}
\usepackage{algpseudocode}
\floatstyle{ruled}
\newfloat{algorithm}{tbp}{loa}
\floatname{algorithm}{Algorithm}
\newfloat{procedure}{tbp}{lop}
\floatname{procedure}{Procedure}
\usepackage{url}
\usepackage{xurl}
\usepackage{placeins}
\usepackage{balance}
\definecolor{ReferenceBlue}{RGB}{0,70,140}
\usepackage[colorlinks=true,linkcolor=ReferenceBlue,citecolor=ReferenceBlue,
  urlcolor=ReferenceBlue,hypertexnames=false]{hyperref}
\usepackage[capitalize,noabbrev]{cleveref}
\usepackage{autonum}
\newcommand{\R}{\mathbb R}
\newcommand{\ip}[2]{\langle #1,#2\rangle}
\newcommand{\norm}[1]{\lVert #1\rVert_2}
\newcommand{\dnorm}[1]{\lVert #1\rVert_*}
\newcommand{\fnorm}[1]{\lVert #1\rVert_F}
\newcommand{\Ball}{\mathbb B_2}
\newcommand{\Mnorm}{\mathcal M}
\newcommand{\Lbest}{L_{\mathrm{best}}}
\newcommand{\ind}{\mathbf 1}
\newcommand{\eps}{\varepsilon}
\DeclareMathOperator{\dist}{dist}
\DeclareMathOperator{\spanop}{span}

\DeclareMathOperator{\argmin}{arg\,min}
\DeclareMathOperator{\argmax}{arg\,max}
\newtheorem{theorem}{Theorem}
\newtheorem{lemma}[theorem]{Lemma}

\newtheorem{corollary}[theorem]{Corollary}
\theoremstyle{definition}
\newtheorem{assumption}{Assumption}
\newtheorem{example}{Example}

\hypersetup{pdftitle={Dimension-Adaptive Accuracy Certificates for Convex Optimization},
  pdfauthor={Egor Gladin},pdfsubject={Convex optimization and accuracy certificates}}
\begin{document}
\runningtitle{Dimension-Adaptive Accuracy Certificates}
\runningauthor{Egor Gladin}
\twocolumn[
\aistatstitle{Dimension-Adaptive Accuracy Certificates\\for Convex Optimization}
\aistatsauthor{Egor Gladin}
\aistatsaddress{HSE University}
]
\begin{abstract}
Many convex optimization methods use affine lower models to certify that a candidate solution is near optimal. Keeping full gradients can be costly when the informative directions occupy a small, initially unknown subspace. We develop accuracy certificates that store projected affine minorants together with bounds on the discarded components. The certificates remain valid as the subspace, queries, and certificate weights are chosen adaptively. For quadratic regularization, we derive an exact correction that accounts for when each discarded component was recorded, and a compact problem for selecting certificate weights. We also characterize the strongest lower bound supported by the compressed records when an ambient direction remains unused. The certificates can be used as stopping criteria for methods that provide valid affine minorants. We also analyze an adaptive cutting-plane construction whose oracle bound depends on the discovered dimension. Experiments on kernel SVM and $\nu$-SVR problems show reduced certificate communication and stopping times under an emulated bandwidth constraint.

\end{abstract}
\section{Introduction}
\label{sec:introduction}
The problem of developing a reliable stopping criterion arises when an optimization method cannot bound the remaining gap to the minimum.
For convex problems, affine lower models resolve this difficulty. For a convex function $h$, an exact subgradient $a_i\in\partial h(x_i)$ at a query point $x_i$ gives the affine minorant
\begin{equation}
 h(x)\ge b_i+\ip{a_i}{x},\qquad
 b_i=h(x_i)-\ip{a_i}{x_i}.
 \label{eq:exact-cut}
\end{equation}
Taking a convex combination of these minorants and minimizing the resulting model over the feasible domain gives a lower bound on the minimum. Subtracting this bound from a feasible objective value produces an observable upper bound on the optimality gap \citep{teo2010}.
Accuracy certificates formalize this principle for execution protocols \citep{nemirovski2010,gladin2025}.

The cost of retaining that information can be substantial. An archive of $N$ slopes in $\R^d$ uses $Nd$ numbers even when the observed slopes lie close to a much smaller space. Projecting the slopes reduces storage, but also changes the affine models; the resulting projected gap then need not certify the original problem.
Greedy reduced-basis methods build approximation spaces incrementally \citep{devore2013}. For certification, an additional challenge arises when the subspace used to approximate observed subgradients expands: a component discarded at an earlier iteration may lie along a direction added later.

We study the following question: \emph{what information must accompany adaptively compressed affine minorants to preserve a trustworthy, and as sharp as possible, ambient stopping gap?} Our answer is to store the projected part of each slope together with a verified set containing the omitted residual. Following the uncertainty-set viewpoint of robust optimization \citep{bental1998}, the construction accounts for every residual in that set before minimizing the model.
The certificate allows arbitrary convex combinations of the recorded affine minorants and remains valid even when query points depend on the evolving compression space or the assumed low-dimensional structure is inaccurate.

For an objective with a known quadratic regularizer, the answer is especially explicit. Suppose the retained spaces grow by appending orthonormal directions. Each raw cut stores its projected coordinates, the norm of its omitted slope, and the dimension of the space when it was recorded.
A short backward recursion computes the exact largest possible norm of a weighted aggregate slope. The resulting certificate weights are selected by a compact second-order-cone weight-selection problem.

\paragraph{Contributions.}
\begin{itemize}
 \item
 We generalize affine-minorant certificates to slopes represented by a
 retained component and a verified uncertainty set for the residual.
 \item
 For a quadratic regularizer and a growing orthonormal basis, we give an exact
 recursion for evaluating the certificate and derive a second-order-cone
 weight-selection problem that yields an optimal certificate under stated
 assumptions.
 \item
 We quantify storage and evaluation costs and describe how to use the
 certificates as auxiliary stopping criteria.
 \item
 On a Euclidean ball, we analyze a construction that doubles its working
 dimension as needed. Its bound on objective-oracle calls depends on
 discovered dimension, assuming a specified volume-reduction procedure.
 \item
 On kernel SVM and $\nu$-SVR problems, we demonstrate that compressed
 certificates can reduce communication and stopping time for three
 first-order methods under an emulated bandwidth constraint.
\end{itemize}

\paragraph{Relation to prior work.}
The weighted-minorant principle is established. \citet{nemirovski2010} develop accuracy certificates for convex minimization, saddle problems, and variational inequalities. \citet{gladin2025} extend them to inexact oracles and construct certificates for polytope cutting-plane methods. Fenchel views of first-order methods appear in \citet{bach2015,gutman2023}, while \citet{burns2026} distinguish primal convergence from convergence of a computable certificate. We use these foundations; the focus here is the geometric information that survives adaptive slope compression.

Weighted affine models, cut aggregation, and inexact minorants are established tools in convex optimization \citep{teo2010,oliveira2014,oliveira2016,parshakova2024}. Our focus is preserving a lower bound for the target objective after adaptive slope compression. Robust optimization supplies uncertainty-set calculus \citep{bental1998}, and exact worst-case oracle analysis emphasizes the role of the information retained \citep{taylor2017}. The historical sets used here admit both an exact aggregate calculation and a restricted full-record minimax equality.

Low-rank online learning and greedy reduced bases adapt to unknown spaces under different performance criteria \citep{hazan2016,devore2013}. Error feedback addresses compressed updates \citep{karimireddy2019}; matrix sketching and spectral bundling reduce storage of matrix iterates or problem-specific models \citep{yurtsever2017,angell2024}. Our compressed object is a vector affine minorant, and the guarantee is a lower bound for the ambient objective. Our adaptive cutting-plane construction combines established fixed-dimensional volume reduction \citep{anstreicher1997,jiang2020} with online discovery of informative directions.

\paragraph{Notation.}
We use the Euclidean inner product and norm unless stated otherwise. For a linear subspace $S$, $\Pi_S$ denotes the orthogonal projection onto $S$. The simplex is $\Delta_N=\{\alpha\in\R^N:\alpha\ge0,\ \ind^\top\alpha=1\}$. A function's convex conjugate is $\Psi^*(y)=\sup_x\{\ip{y}{x}-\Psi(x)\}$. A vector $a\in\partial_\delta h(x)$ is a $\delta$-subgradient of $h$ at $x$, meaning that $h(y)\ge h(x)+\ip{a}{y-x}-\delta$ for every $y$. For a set $B$, its indicator $I_B$ is zero on $B$ and infinite outside it. For sets $K_i$, $\sum_i\alpha_iK_i$ denotes their weighted Minkowski sum: the set of vectors $\sum_i\alpha_i r_i$ with $r_i\in K_i$.

\section{The problem and aggregate certificates}
\label{sec:baseline}
Let $X\subseteq\R^d$ be nonempty and convex, and consider
\begin{equation}
 F_* = \inf_{x\in X}\{F(x):=h(x)+\psi_0(x)\}\in\R,
 \label{eq:baseline-problem}
\end{equation}
where $h$ is convex and $\psi_0$ is the known part of the objective.
Choose a convex outer domain $B\supseteq X$ such that
$\Psi=\psi_0+I_B$ is proper, closed, and convex.

Given $N$ affine minorants as in \eqref{eq:exact-cut}, any
$\alpha\in\Delta_N$ defines
\begin{equation}
 b^\alpha=\sum_{i=1}^N\alpha_i b_i,\qquad
 a^\alpha=\sum_{i=1}^N\alpha_i a_i.
 \label{eq:baseline-aggregate}
\end{equation}
satisfying
\begin{equation}
\begin{aligned}
 F_*&\ge b^\alpha+\inf_{x\in X}
       \{\psi_0(x)+\ip{a^\alpha}{x}\}\\
    &\ge b^\alpha-\Psi^*(-a^\alpha)
       =:L_{\mathrm{base}}(\alpha).
\end{aligned}
 \label{eq:baseline-certificate}
\end{equation}
For a feasible incumbent $\widehat x$ and $U \geq F(\widehat x)$,
the baseline certificate gives
$F(\widehat x)-F_*\le U-L_{\mathrm{base}}(\alpha)$.
Maximizing $L_{\mathrm{base}}(\alpha)$ over $\Delta_N$ gives the lower bound used
in the BMRM stopping test \citep[Theorem~2]{teo2010}.
When $X=\R^d$ and $\psi_0(x)=\lambda\norm{x}^2/2$, $\lambda>0$,
the certificate is especially simple:
\begin{equation}
 L_{\mathrm{base}}(\alpha)
 =b^\alpha-\frac{\norm{a^\alpha}^2}{2\lambda}\le F_*.
 \label{eq:baseline-quadratic}
\end{equation}

\paragraph{Compressed records.}
Instead of retaining the full slope $a_i$, store an intercept $b_i$,
a retained slope $p_i$, and a verified uncertainty set $K_i$ for the
discarded residual $r_i := a_i - p_i$.
Thus, the record $(b_i,p_i,K_i)$ asserts that there exists a residual $r_i\in K_i$ such that
\begin{equation}
 h(x)\ge b_i+\ip{p_i+r_i}{x}\quad\text{for every }x\in X.
 \label{eq:baseline-record}
\end{equation}
For $\alpha\in\Delta_N$, define
\begin{equation}
 p^\alpha=\sum_i\alpha_i p_i,\qquad
 K^\alpha=\sum_i\alpha_iK_i.
 \label{eq:compressed-aggregate}
\end{equation}
The actual aggregate slope $a^\alpha$ belongs to $p^\alpha+K^\alpha$.
The following theorem generalizes \eqref{eq:baseline-certificate} to
compressed records by accounting for residual uncertainty.
\begin{theorem}
\label{thm:baseline-robust}
Suppose the $N$ records satisfy \eqref{eq:baseline-record}, and let
$U\ge F(\widehat x)$ for a feasible $\widehat x\in X$.
For every $\alpha\in\Delta_N$,
\begin{gather}
 L(\alpha)=b^\alpha-
 \sup_{r\in K^\alpha}\Psi^*(-p^\alpha-r)\le F_*,
 \label{eq:baseline-robust-bound}\\
 0\le F(\widehat x)-F_*\le U-L(\alpha).
 \label{eq:baseline-gap}
\end{gather}
\end{theorem}
The proof is in Appendix~\ref{app:robust-proof}.
The bounds remain valid under adaptive choices of compression spaces and retained records, without requiring low-dimensional structure.
If the supremum in \eqref{eq:baseline-robust-bound} is difficult
to compute, replacing it by a certified upper bound remains safe.

\paragraph{An exact certificate for a growing subspace}
We now consider unconstrained minimization with the quadratic regularizer
$\psi_0(x)=\lambda\norm{x}^2/2$, where $\lambda>0$.
Let $q_1,\ldots,q_k$ be an
append-only orthonormal basis used to compress the observed slopes. Write
$Q=[q_1\ \cdots\ q_k]$ and define
\begin{equation}
 S_j=\spanop(q_1,\ldots,q_j),\qquad S_0=\{0\}.
 \label{eq:baseline-subspaces}
\end{equation}
For $j=0,\ldots,k$, let $\mathcal B_j=\{r\in S_j^\perp:\norm{r}\le1\}$ be
the unit ball used to model a residual left after projection onto $S_j$.
A raw slope $a_i$ projected when the basis has size $m_i$ is represented by
the stored tuple $(v_i,\eta_i,m_i)$, satisfying
\[
p_i :=\Pi_{S_{m_i}}a_i=Qv_i,\qquad
 \eta_i\ge\norm{a_i-p_i}
\]
Then $r_i \in \eta_i \mathcal B_{m_i}$.
For $\alpha\in\Delta_N$, let $v^\alpha=\sum_i\alpha_i v_i$. Then
\[
   a^\alpha = Q v^\alpha + \sum_{i=1}^N \alpha_i r_i,
\]
where
\[
 \sum_{i=1}^N\alpha_i r_i
 \in\sum_{j=0}^k
 \left(\sum_{i:m_i=j}\alpha_i\eta_i\right)
 \mathcal B_j=K(c^\alpha),
\]
where $c^\alpha_j:=\sum_{i:m_i=j}\alpha_i\eta_i$ and
$K(c):=\sum_{j=0}^k c_j\mathcal B_j$ for
$c\in\R_+^{k+1}$.

To evaluate the certificate, we need the
largest slope norm compatible with the stored information. The following
procedure computes it by working backward through the basis. Each step
combines a retained coordinate with the orthogonal components already
accounted for, then adds the next residual radius to allow for possible
alignment.

\begin{procedure}[H]
\caption{Historical-residual norm computation}
\label{proc:baseline-norm}
\small
\textbf{Input:} $v\in\R^k$, $c\in\R_+^{k+1}$, ambient dimension $d\ge k$.
\begin{enumerate}
\setlength{\itemsep}{2pt}
\item Set $u_k\leftarrow c_k$ if $k<d$, and $u_k\leftarrow0$ otherwise.
\item For $j=k,k-1,\ldots,1$, set
\begingroup
\setlength{\abovedisplayskip}{3pt}
\setlength{\belowdisplayskip}{3pt}
\setlength{\abovedisplayshortskip}{3pt}
\setlength{\belowdisplayshortskip}{3pt}
\[
 u_{j-1}\leftarrow c_{j-1}+\sqrt{v_j^2+u_j^2}.
\]
\endgroup
\item Return $\Mnorm(v,c):=u_0$.
\end{enumerate}
\end{procedure}
If $k=0<d$, this reads $\Mnorm=c_0$. At full rank, $\mathcal B_d=\{0\}$, so $c_d$ has no effect.

\begin{lemma}
\label{lem:baseline-norm}
For all $v\in\R^k$ and $c\in\R_+^{k+1}$,
Procedure~\ref{proc:baseline-norm} returns
\begin{equation}
 \Mnorm(v,c)=\max_{r\in K(c)}\norm{Qv+r}.
 \label{eq:baseline-exact-norm}
\end{equation}
\end{lemma}
The proof is in Appendix~\ref{app:baseline-norm-proof}.

We now use this norm in Theorem~\ref{thm:baseline-robust} to obtain a
certificate for the quadratic regularizer. Choosing the aggregation weights
to maximize this lower bound gives the \emph{weight-selection problem}.

\begin{theorem}
\label{thm:baseline-quadratic}
For the compressed records above, every $\alpha\in\Delta_N$ gives
\begin{equation}
 L_Q(\alpha):=b^\alpha-
 \frac{\Mnorm(v^\alpha,c^\alpha)^2}{2\lambda}\le F_*.
 \label{eq:baseline-quadratic-cert}
\end{equation}
The best bound of this form is the negative optimal value of
\begin{equation}
\begin{aligned}
 \min_{\substack{\alpha\in\Delta_N\\u\in\R_+^{k+1}}}\quad
 &\frac{u_0^2}{2\lambda}-b^\alpha\\
 \text{s.t.}\quad &\norm{(v^\alpha_j,u_j)}\le u_{j-1}-c^\alpha_{j-1}
                         \quad(1\le j\le k),\\
 &u_k\ge c^\alpha_k\quad(k<d).
\end{aligned}
 \label{eq:baseline-weight-master}
\end{equation}
\end{theorem}
The proof is in Appendix~\ref{app:baseline-quadratic-proof}.
Note that  \eqref{eq:baseline-weight-master}  can be represented as a second-order-cone program
(SOCP), with $N$ weight variables and $O(k+1)$ additional variables and
cone constraints.

The recording dimensions must remain fixed as the basis grows.
Treating an old residual as orthogonal to the current space can produce
a false lower bound.
\begin{example}
\label{ex:history}
Let $d=2$, $\lambda=1$, $0<\eta<1$, and
\[
 h(x)=\max\{\eta x_1+(\eta^2-1)/2,\ x_1\}.
\]
Discard the first slope at dimension zero with radius $\eta$, then
retain the second slope along the first coordinate axis.
For $\alpha=(1/2,1/2)$, we have $v^\alpha=1/2$ and
$c^\alpha=(\eta/2,0)$, so $\Mnorm(v^\alpha,c^\alpha)=(1+\eta)/2$ and
\[
 L_Q(\alpha)=\frac{\eta^2-2\eta-3}{8}=F_*.
\]
Relabeling the old residual as orthogonal to the current space instead
gives the norm $\sqrt{1+\eta^2}/2$ and reports
$(\eta^2-3)/8$, exceeding $F_*$ by $\eta/4$.
The calculation is in Appendix~\ref{app:counterexamples}.
\end{example}

The next result characterizes the strongest lower bound supported by the
compressed records alone. 
\begin{theorem}
\label{thm:baseline-minimax}
Assume $k<d$. For independently chosen $r_i\in\eta_i\mathcal B_{m_i}$,
define
\[
 h_r(x)=\max_i\{b_i+\ip{Qv_i+r_i}{x}\}.
\]
Then
\begin{equation}
\begin{aligned}
 &\inf_{r_i\in\eta_i\mathcal B_{m_i}}\inf_{x\in\R^d}
 \left\{h_r(x)+\frac\lambda2\norm{x}^2\right\} =\max_{\alpha\in\Delta_N}L_Q(\alpha).
\end{aligned}
 \label{eq:baseline-minimax-equality}
\end{equation}
\end{theorem}
The proof is in Appendix~\ref{app:minimax-proof}.
Thus the optimized certificate is the largest lower bound valid for every
convex loss compatible with the stored minorant inequalities alone.

\paragraph{Computational cost and storage.}
For fixed weights and $k\ge1$, evaluating the baseline certificate
\eqref{eq:baseline-quadratic} costs $O(Nd)$ arithmetic operations.
Evaluating the compressed certificate \eqref{eq:baseline-quadratic-cert}
costs $O(Nk)$, including $O(k)$ for the recursion; forming each compressed
record from a new dense slope during the underlying optimization requires
an additional $O(dk)$ projection cost. Storing the cut archive requires
$O(Nd)$ numbers for the baseline and $O(Nk)$ for the compressed
representation, with an additional $O(dk)$ for the basis. Optimizing the
weights gives a quadratic program for the baseline and an SOCP for the
compressed certificate; their solution costs depend on the solver and
implementation.

\paragraph{Use with optimization methods.}
Compressed certificates can accompany any method that supplies valid
affine minorants of the convex loss, including gradient and quasi-Newton
methods. For methods such as L-BFGS that do not otherwise retain the full
slope history, compressed records provide an economical auxiliary archive
for certification. At each query, compute the intercept before compression
and store $(b_i,v_i,\eta_i,m_i)$ with a verified residual bound.
Any simplex weights are valid; they can optionally be optimized via
\eqref{eq:baseline-weight-master}. Evaluate
\eqref{eq:baseline-quadratic-cert} and use $U-L_Q(\alpha)$ as the stopping
gap. Compression is most useful when many slopes must be retained in a
large ambient dimension while their discovered span remains small.

\section{A cutting-plane method with adaptive dimension}
\label{sec:cutting-plane}
We now use the compressed records to choose queries and certify an averaged
output on a Euclidean ball. The method discovers directions as needed and
requires no function values or quadratic regularizer.

\paragraph{An averaged-point certificate.}
Let $X=R\Ball^d$, $R>0$, and let $F$ be convex with
$F_*=\inf_{x\in X}F(x)\in\R$. At each objective query
$x_i\in\operatorname{int}X$, the oracle returns a $\delta_i$-subgradient
$g_i$ of $F$ relative to $X$, satisfying
\begin{equation}
\begin{aligned}
 F(z)&\ge F(x_i)+\ip{g_i}{z-x_i}-\delta_i\quad(z\in X),\\
 0\le\delta_i&\le\delta,\qquad\norm{g_i}\le G,
\end{aligned}
 \label{eq:ball-oracle}
\end{equation}
where $G>0$. Let $S$ denote the current discovered space, which contains
the query point $x_i$, and let $Q$ be its orthonormal basis matrix.
Set $r_i:=g_i-\Pi_Sg_i$ and apply the discovery
and compression rule
\begin{equation}
\begin{gathered}
 Q\gets\begin{bmatrix}Q & \dfrac{r_i}{\norm{r_i}}\end{bmatrix}
 \quad\text{if }\norm{r_i}>\tau,\\
 (p_i,\eta_i)\gets
 \begin{cases}
  (g_i,0), & \norm{r_i}>\tau,\\
  (\Pi_Sg_i,\norm{r_i}), & \norm{r_i}\le\tau.
 \end{cases}
\end{gathered}
\label{eq:cp-compression}
\end{equation}
In both cases $\ip{g_i-p_i}{x_i}=0$.

As in the preceding section, store $(v_i,\eta_i,m_i)$, where $m_i$ is
the basis size at recording and $p_i=Qv_i$ in the current basis.
For simplex weights $\alpha$ on the retained records, let
$\bar x^\alpha:=\sum_i\alpha_i x_i$ and set
\[
 v^\alpha:=\sum_i\alpha_i v_i,\qquad
 c^\alpha_j:=\sum_{i:m_i=j}\alpha_i\eta_i.
\]
\begin{lemma}
\label{lem:ball-certificate}
The averaged point $\bar x^\alpha$ satisfies
\begin{gather}
 F(\bar x^\alpha)-F_*\le C(\alpha),\\
 C(\alpha):=\sum_i\alpha_i(\ip{p_i}{x_i}+\delta_i)
                  +R\Mnorm(v^\alpha,c^\alpha).
 \label{eq:ball-certificate}
\end{gather}
\end{lemma}
The proof is in Appendix~\ref{app:cp-validity}.
Minimizing the convex function $C$ over the simplex selects both the output
and its gap certificate; $\Mnorm$ is evaluated by
Procedure~\ref{proc:baseline-norm}.

\paragraph{The cutting-plane method.}
Our method uses an existing fixed-dimensional volumetric cutting-plane
method as a subroutine, such as Vaidya's method
\citep{anstreicher1997,jiang2020}.
These methods maintain polytopes, called \emph{localizers}, defined by
the retained cuts, and compute approximate volumetric centers at which
new cuts are requested. Cut insertion and constraint deletion guarantee
an overall reduction in localizer volume.

Our algorithm proceeds in phases. Within each phase, the cutting-plane
subroutine operates in a fixed \emph{working dimension} $q$.
We use the following interface.

\begin{assumption}[Fixed-dimensional cutting-plane primitive]
\label{ass:primitive}
Starting from $[-R,R]^q$, the primitive processes nonzero central cuts,
possibly retaining outward relaxations, and keeps $O(q)$ constraints.
Within $O(q\log(2+qR/r))$ updates, its localizer has volume smaller than
a $q$-dimensional radius-$r$ ball.
\end{assumption}

Let $m$ denote the number of basis vectors discovered so far and let
$y\in\R^q$ be the center returned by the primitive.
Only the first $m$ working coordinates correspond to discovered basis
vectors; the remaining coordinates are reserved for new directions.
The unassigned coordinates of $y$ remain zero, a property ensured by
the symmetry of the volumetric barrier and the standard constraint-deletion
rule; see Appendix~\ref{app:primitive}.
Thus the ambient query $x_i=Qy_{1:m}$ lies in the current discovered space.
When a new direction is discovered, it occupies a reserved coordinate.
If none remains, the algorithm doubles the working dimension,
clipped at $d$, and starts a new phase, keeping the basis and
discarding the previous phase's records.

Algorithm~\ref{alg:adaptive-cp} uses this primitive, denoted by
$\mathrm{CP}$. The primitive receives central cuts $\ip{a}{z-y}\le0$,
possibly retaining outward relaxations. Let $P$ index the retained
objective records and $\Delta_P$ be their simplex. Records
are linked to their cuts and removed when those constraints are deleted.
Their old coordinates are padded with zeros as the basis grows, while
their recording dimensions remain fixed.

\begin{algorithm}[tb]
\caption{Dimension-adaptive cutting-plane method}
\label{alg:adaptive-cp}
\small
\algtext*{EndIf}
\algtext*{EndLoop}
\begin{algorithmic}[1]
\Require $R,\delta,G$, discovery threshold $\tau$, tolerance $\zeta$
\State Initialize an empty basis and record set; set $q\gets1$
\State Initialize CP from $[-R,R]^q$
\Loop
  \State Obtain a CP center $y$ and synchronize the records
  \State $\alpha\gets\argmin_{\beta\in\Delta_P}C(\beta)$
  \If{$C(\alpha)\le\delta+R\tau+\zeta$}
    \State \Return $(\bar x^\alpha,C(\alpha))$
  \EndIf
  \If{$\norm{y}\ge R$}
    \State Add a ball-separation cut; \textbf{continue}
  \EndIf
  \State Query $(g_i,\delta_i)$ at $x_i=Qy_{1:m}$
  \State Apply \eqref{eq:cp-compression} to $g_i$
  \If{$m>q$}
    \State $q\gets\min\{2q,d\}$; restart the phase; \textbf{continue}
  \EndIf
  \State \textbf{if} $p_i=0$, \Return $(x_i,\delta_i+R\eta_i)$
  \State Store the record and add its projected central cut
\EndLoop
\end{algorithmic}
\end{algorithm}
Weight selection and the stopping test are skipped if no records remain.
Synchronizing records means removing those whose
cuts have been deleted. Ball-separation cuts require no objective call.
Appendix~\ref{app:cp-algorithm} gives the detailed pseudocode.

\begin{theorem}
\label{thm:cp}
Under Assumption~\ref{ass:primitive}, for $0<\zeta\le RG$ and
$0\le\tau\le G$, Algorithm~\ref{alg:adaptive-cp}, with exact weight
selection, returns an averaged point and certificate satisfying
\begin{equation}
 F(\bar x^\alpha)-F_*\le C(\alpha)\le\delta+R\tau+\zeta.
 \label{eq:cp-accuracy}
\end{equation}
If it discovers $m\ge1$ directions, then
\begin{equation}
\begin{aligned}
 N_{\mathrm{oracle}}&=O\!\left(m\log\left(2+\frac{mRG}{\zeta}\right)\right),\\
 \text{storage}&=O(dm+m^2)\text{ real numbers}.
\end{aligned}
 \label{eq:cp-cost}
\end{equation}
\end{theorem}
The proof and an extension to approximate weight selection are in
Appendix~\ref{app:cp-proof}. These are exact-arithmetic bounds;
$N_{\mathrm{oracle}}$ counts objective-oracle calls.

\paragraph{Low-dimensional gradient structure.}
The cost depends on the directions actually discovered, which can be
bounded using exact or approximate structure in the observed gradients.
\begin{corollary}
\label{cor:width}
If the returned gradients lie in a fixed $k$-dimensional space, then
$m\le k$. If $\tau>0$, $k\ge1$, and every discovery-triggering gradient
is within $\eta\le\tau/2$ of such a space, then
\begin{equation}
 m\le\min\{d,\lceil3k(1+\log(G/\tau))\rceil\}.
 \label{eq:width-bound}
\end{equation}
\end{corollary}
The proof is in Appendix~\ref{app:width}.
For $0<\eps\le RG$, taking $\tau=\eps/(2R)$ and $\zeta=\eps/2$ gives accuracy
$\delta+\eps$ with approximation scale $\eps/(4R)$; exact structure
permits $\tau=0$ and $\zeta=\eps$.
Gradient components already in the query space create no new directions,
as with a quadratic regularizer.
Appendix~\ref{app:width} discusses this case and why the approximation
scale must be smaller than the discovery threshold.

\section{Experiments}
\label{sec:experiments}
The code for both experiments is publicly available at
\url{https://github.com/egorgladin/adaptive-certificates}.
\subsection{Compressed certificates with limited communication}
\label{sec:channel-experiments}
We study whether compressed certificates can provide useful stopping rules
when communication is expensive. The motivating setting is a GPU worker
that generates optimization queries and sends records to a separate
certification service over a channel with limited effective bandwidth.
We emulate this constraint by adding a delay proportional to each message's
size. The experiment tests the cost of certification along a given
optimization trajectory.

\begin{table*}[!t]
\centering\small
\caption{Stopping at $U-L\le10^{-3}F(0)$ with an emulated certificate
channel of 1 MiB/s. Times are medians of three paired repetitions.
$N_{\rm F}/N_{\rm C}$ counts exact function--subgradient queries,
$k$ is the final compressed basis dimension, and $S_{\rm F}/S_{\rm C}$
is the ratio of total communicated bytes, including common scalars.
Single-cut does not reach the target on any instance within 2000 queries.}
\label{tab:channel-main}
\setlength{\tabcolsep}{5pt}
\begin{tabular}{llrrrcrr}
\toprule
Data & Method & Full (s) & Comp. (s) & Saving & $N_{\rm F}/N_{\rm C}$ & $k$ & $S_{\rm F}/S_{\rm C}$\\
\midrule
\multirow{3}{*}{Abalone} & SG & 10.17 & 3.96 & 61.1\% & 384/385 & 83 & 26.9\\
 & Mom & 2.27 & 0.80 & 64.7\% & 96/96 & 70 & 41.9\\
 & Ada & 4.68 & 1.77 & 62.1\% & 192/192 & 105 & 26.7\\
\midrule
\multirow{3}{*}{Airfoil} & SG & 11.32 & 5.49 & 51.5\% & 597/597 & 76 & 14.6\\
 & Mom & 1.45 & 0.87 & 40.2\% & 96/96 & 64 & 22.3\\
 & Ada & 2.47 & 1.24 & 49.7\% & 160/160 & 62 & 19.7\\
\midrule
\multirow{3}{*}{Adult} & SG & 16.32 & 14.97 & 8.3\% & 612/714 & 256 & 8.2\\
 & Mom & 3.89 & 1.68 & 56.9\% & 160/160 & 160 & 24.0\\
 & Ada & 7.37 & 4.28 & 42.0\% & 306/320 & 256 & 12.3\\
\midrule
\multirow{3}{*}{Breast Cancer} & SG & 3.37 & 2.21 & 34.4\% & 376/376 & 59 & 8.8\\
 & Mom & 2.18 & 1.59 & 27.1\% & 246/246 & 57 & 8.7\\
 & Ada & 12.77 & 8.86 & 30.6\% & 850/850 & 73 & 5.4\\
\bottomrule
\end{tabular}
\end{table*}

\paragraph{Problems and protocol.}
We use RBF $\nu$-SVR on Abalone and Airfoil, and RBF hinge-loss SVM on Adult
and Breast Cancer Wisconsin (Diagnostic)
\citep{abalone_data,airfoil_data,adult_data,breast_cancer_data}.
Both objectives have a quadratic regularizer and an unregularized intercept,
which is minimized out at each query. The RBF width is fixed to the median
pairwise training distance. We retain the full kernel factor: compression
is applied to certificate records, without approximating the training
objective.

The query methods are subgradient descent (SG), momentum (Mom), and
diagonal AdaGrad (Ada). We compare three stopping certificates:
\emph{Single-cut}, the best single-cut strong-convexity bound;
\emph{full}, the optimized aggregate of all full slopes; and
\emph{compressed}, the historical-residual certificate
\eqref{eq:baseline-quadratic-cert} with weights selected by
\eqref{eq:baseline-weight-master}.
Both archive masters use all recorded cuts, without working sets.
The certificates do not affect the queries, and all variants retain the
single-cut bound. We stop when the incumbent gap satisfies
$U-L\le\eps=10^{-3}F(0)$, with a cap of 2000 function--subgradient queries.
Model and method parameters are selected before the main timing runs;
Appendix~\ref{app:channel-experiments} gives the details.

\begin{figure*}[!t]
\centering
\includegraphics[width=\textwidth]{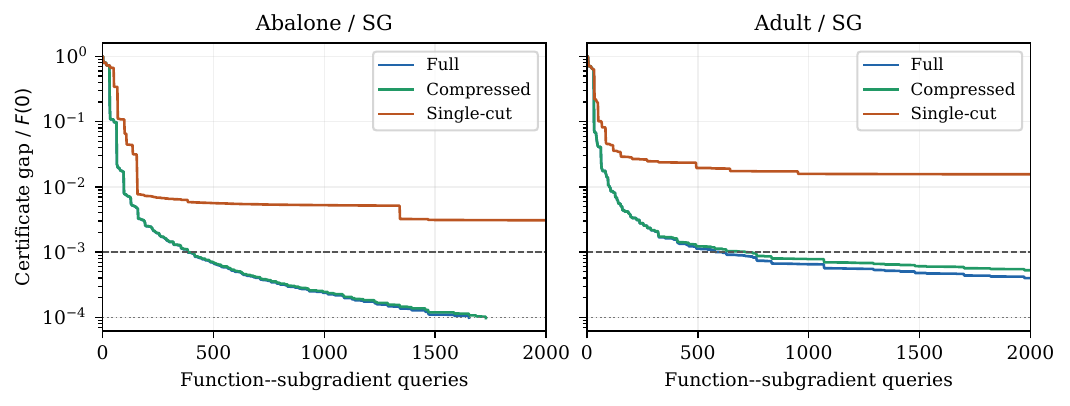}
\caption{Certificate gaps along continued subgradient trajectories.
Full and compressed reach the main target (dashed line) with similar
query counts, while single-cut remains substantially weaker. The dotted line
marks a target ten times tighter: both archives reach it on Abalone,
while neither reaches it on Adult within 2000 queries.}
\label{fig:channel-gaps}
\end{figure*}

\paragraph{Results.}
At 1 MiB/s, compressed reaches the target faster than full in all twelve
problem--method combinations, saving 8.3--64.7\% of the median time
(Table~\ref{tab:channel-main}). All 36 paired timing comparisons favor
compression. Total communication decreases by factors of 5.4--41.9,
while query counts agree in nine combinations. The largest query overhead
is on Adult/SG: 714 queries instead of 612.

The archive matters for certificate strength
(Figure~\ref{fig:channel-gaps}). For example, after 2000 queries on
Adult/SG, the incumbent value is 0.2853, close to an independently
computed feasible reference value of 0.2851, but the single-cut gap is
0.0078, compared with a target of 0.00050.
Aggregating slopes supplies a substantially stronger stopping bound.

\paragraph{Where the gain comes from.}
Communication savings can offset projection, weight selection, and
additional queries. On Adult/SG, compression saves about 8.40 seconds of
channel waiting while adding 7.05 seconds of other work, yielding a net
benefit even with the largest increase in query count.
On Abalone/SG, communication savings account for 5.79 seconds, while
the remaining work is also slightly cheaper.
The historical correction keeps the compressed gaps close to the full
gaps as the archive grows, so reducing communication does not require
giving up the prescribed accuracy. These results demonstrate a useful
role for the theory as a certificate layer for several standard
first-order methods. Bandwidth and accuracy sensitivity are reported
in Appendix~\ref{app:channel-experiments}.

\subsection{A numerical illustration of dimension adaptation}
\label{sec:adaptive-cp-experiment}
We illustrate the theoretical contribution of Theorem~\ref{thm:cp}: an
oracle bound governed by discovered dimension, together with a computable
accuracy certificate. A controlled synthetic problem lets us vary ambient
and hidden dimensions independently and check the certificate against a
known optimum. We therefore compare coordinate choices for the same
cutting-plane primitive and report objective-oracle calls, the complexity
measure bounded by the theorem.

\paragraph{Problem and comparisons.}
Inspired by the nonsmooth example in
\citet[Section~3.2.1]{nesterov2018}, we embed a maximum of coordinates
in an unknown subspace and normalize its scale.
Let $U=[u_1,\ldots,u_k]\in\R^{d\times k}$ have orthonormal columns. We solve
\begin{equation}
 \begin{gathered}
  \min_{\norm{x}\le2}F_{d,k}(x),\\
  F_{d,k}(x)=\sqrt{k}\max_{1\le j\le k}\ip{u_j}{x}
                         +\frac12\norm{x}^2.
 \end{gathered}
 \label{eq:cp-experiment-objective}
\end{equation}
The normalization gives $x_*=-U\ind/\sqrt{k}$, $\norm{x_*}=1$,
$F_*=-1/2$, and $F(0)-F_*=1/2$ for every $d,k$.
The oracle returns the ambient subgradient $g(x)=\sqrt{k}u_{j_*}+x$,
choosing the first maximizing index. The term $x$ lies in the query space,
so the adaptive method discovers at most $k$ directions.

All three variants use a central-cut volumetric primitive with
minimum-leverage deletion \citep{anstreicher1998,blanchard2023} and the
averaged-point certificate \eqref{eq:ball-certificate}.
\emph{Adaptive} follows Algorithm~\ref{alg:adaptive-cp} without knowing $U$;
\emph{known subspace} is given $U$ and works in $k$ coordinates;
\emph{ambient} works in the original $d$ coordinates without compression.
We vary $d\in\{64,256,1024\}$ and $k\in\{2,4,8,16\}$ at $\eps=10^{-4}$,
then vary $\eps$ from $10^{-2}$ to $10^{-6}$ at $d=256,k=8$.
The ambient comparison uses $d\in\{16,32,64\}$ and $k=4$.
Appendix~\ref{app:cp-experiment} gives the protocol and complete counts.

\begin{figure*}[!t]
 \centering
 \includegraphics[width=\textwidth]{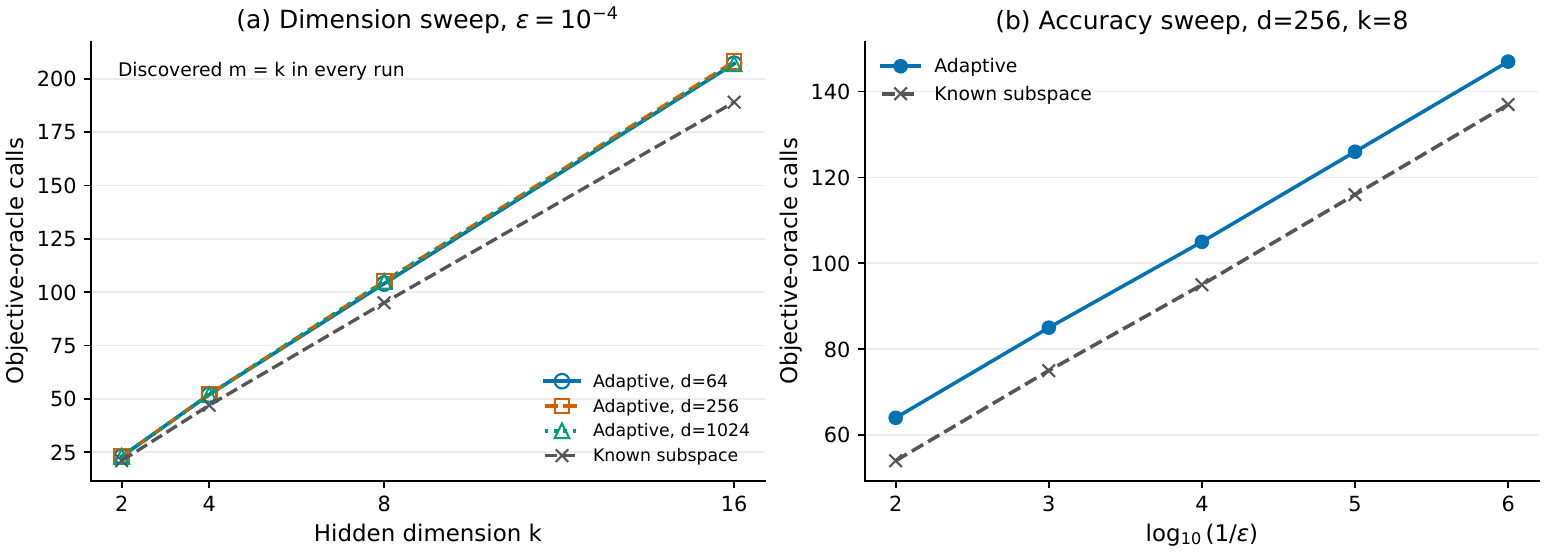}
 \caption{Objective-oracle calls to certified accuracy.
 (a) Changing the ambient dimension has almost no effect on the adaptive
 counts; all runs discover $m=k$ directions.
 (b) At $d=256,k=8$, each tenfold decrease in $\eps$ adds 20--21 calls.
 The known-subspace variant uses the same primitive with $U$ supplied in
 advance.}
 \label{fig:adaptive-cp-experiment}
\end{figure*}

\paragraph{Results.}
All 31 runs reach every requested accuracy, giving 43 certified endpoints.
In every adaptive run, $m=k$. At fixed $k$, increasing $d$ from 64 to 1024
changes the call count by at most one
(Figure~\ref{fig:adaptive-cp-experiment}(a)).
Discovering the space costs about 10\% more calls than being given it.
The accuracy sweep is consistent with logarithmic dependence on
$1/\eps$ (Figure~\ref{fig:adaptive-cp-experiment}(b)).
Against the ambient variant, adaptation reduces calls by factors of
2.58, 4.83, and 9.08 at $d=16,32,64$, respectively
(Table~\ref{tab:cp-ambient}).
Every checked averaged output satisfies $F(\bar x)-F_*\le C$, and every
stopping certificate satisfies $C\le\eps$.
These observations illustrate dimension adaptation and certified stopping
in the exact-rank setting; they do not establish the benefit of the
historical correction when discarded components are appreciable.

\begin{table}[!tb]
 \centering\small
 \caption{Calls to certify $\eps=10^{-4}$ at $k=4$, using the same
 volumetric primitive in three coordinate systems.}
 \label{tab:cp-ambient}
 \begin{tabular}{rrrr}
  \toprule
  $d$ & Known subspace & Adaptive & Ambient\\
  \midrule
  16 & 47 & 52 & 134\\
  32 & 47 & 52 & 251\\
  64 & 47 & 52 & 472\\
  \bottomrule
 \end{tabular}
\end{table}

\newpage
\balance
\section{Discussion}
The information retained with a compressed minorant determines what it can
certify about the ambient problem. For nested spaces and a known quadratic
regularizer, the omitted norm and recording dimension give a compact description
of each raw residual. The backward recursion converts that description into
a global lower bound; comparing it with an incumbent value gives a stopping
gap. When an unused ambient direction remains, optimizing the weights
exhausts the information in the minorant records alone.

The certificates can be evaluated independently of the query rule.
Our adaptive cutting-plane method gives an explicit way to couple
certification with an oracle bound governed by discovered dimension,
under the stated fixed-dimensional volume-reduction assumption.

The communication experiments show that compressed records can reduce
transmission costs and stopping time while retaining useful aggregate
certificates along first-order optimization trajectories. The gain
depends on bandwidth, target accuracy, and the costs of projection
and weight selection. A separate numerical experiment illustrates
dimension adaptation: oracle counts depend primarily on the discovered
dimension, with a modest overhead relative to a known-subspace method.
That experiment uses exact-rank gradients. Tests on approximately
low-dimensional trajectories, larger problems, and actual distributed
deployments would further clarify the practical scope of the results.

\clearpage
\nobalance
\subsection*{AI use statement}
Generative AI assisted with language editing, literature searches, proof development, and the implementation of experiments. All AI-assisted mathematical and experimental content was independently verified, and the author takes full responsibility for the manuscript.
\bibliography{references}
\clearpage
\appendix
\onecolumn
\aistatstitle{Dimension-Adaptive Accuracy Certificates for Convex Optimization: Appendices}
\section{Detailed proofs and sharpness examples}
\label{app:proofs}
\subsection{Proof of the robust aggregate certificate}
\label{app:robust-proof}
\begin{proof}[Proof of Theorem~\ref{thm:baseline-robust}]
For each record, choose the fixed actual residual $r_i\in K_i$ from
\eqref{eq:baseline-record}. With $r^\alpha=\sum_i\alpha_i r_i\in K^\alpha$,
the weighted minorants give
$h(x)\ge b^\alpha+\ip{p^\alpha+r^\alpha}{x}$ for every $x\in X$.
Since $X\subseteq B$,
\begin{align}
 F_*&\ge b^\alpha+\inf_{x\in X}
       \{\psi_0(x)+\ip{p^\alpha+r^\alpha}{x}\}\\
    &\ge b^\alpha-\Psi^*(-p^\alpha-r^\alpha)\\
    &\ge b^\alpha-\sup_{r\in K^\alpha}\Psi^*(-p^\alpha-r)
       =L(\alpha).
\end{align}
Subtracting this bound from $U\ge F(\widehat x)$ proves \eqref{eq:baseline-gap}.
The argument holds for each realized collection of valid records and
every simplex vector, regardless of how they were selected.
\end{proof}

\subsection{Proof of Lemma~\ref{lem:baseline-norm}}
\label{app:baseline-norm-proof}
\begin{proof}
For $j=0,\ldots,k$, define the compact sets
\[
 T_j=\sum_{\ell=j+1}^k v_\ell q_\ell+
     \sum_{\ell=j}^k c_\ell\mathcal B_\ell
     \subseteq S_j^\perp.
\]
We show by backward induction that $u_j$ in
Procedure~\ref{proc:baseline-norm} is the largest norm in $T_j$.
At $j=k$, the set is $c_k\mathcal B_k$, whose largest norm is $c_k$
if $k<d$ and zero if $k=d$.

Suppose the claim holds at $j\ge1$. Since $q_j$ is orthogonal to $T_j$,
the largest norm in $v_jq_j+T_j$ is $\sqrt{v_j^2+u_j^2}$.
Moreover,
\[
 T_{j-1}=v_jq_j+T_j+c_{j-1}\mathcal B_{j-1}.
\]
The first two terms lie in $S_{j-1}^\perp$, so the last ball can align
with a maximizing vector and increase its norm by $c_{j-1}$.
If that vector is zero, any vector of norm $c_{j-1}$ in this ball
attains the same value. The triangle inequality gives the matching upper
bound, proving the recursive step. Finally, $T_0=Qv+K(c)$, so its largest
norm is the returned value $u_0=\Mnorm(v,c)$, as claimed.
\end{proof}

\subsection{Proof of Theorem~\ref{thm:baseline-quadratic}}
\label{app:baseline-quadratic-proof}
\begin{proof}
Take $K_i=\eta_i\mathcal B_{m_i}$. Grouping these balls by birth dimension
gives $K^\alpha=K(c^\alpha)$, while $p^\alpha=Qv^\alpha$.
The conjugate of $\lambda\norm{x}^2/2$ is $\norm{y}^2/(2\lambda)$,
so Theorem~\ref{thm:baseline-robust} and Lemma~\ref{lem:baseline-norm}
give \eqref{eq:baseline-quadratic-cert}.

For completeness, writing the Euclidean norm as a maximum of inner
products gives
\begin{equation}
 \Mnorm(v,c)=\max_{\norm{z}\le1}
 \left\{\ip{Qv}{z}+\sum_{j=0}^k c_j\norm{\Pi_{S_j^\perp}z}\right\}.
 \label{eq:norm-support}
\end{equation}
For fixed $z$, each residual ball contributes its support value
$c_j\norm{\Pi_{S_j^\perp}z}$. The displayed expression is a maximum
of linear functions of $(v,c)$, so $\Mnorm$ is convex and nonnegative;
its square is convex as well.

For fixed weights, the constraints in \eqref{eq:baseline-weight-master}
bound each $u_j$ below by its value in Procedure~\ref{proc:baseline-norm},
as follows by backward induction. Choosing exactly those values is
feasible. Since the objective is increasing in $u_0\ge0$, minimizing
over $u$ gives
$\Mnorm(v^\alpha,c^\alpha)^2/(2\lambda)-b^\alpha=-L_Q(\alpha)$.
Minimizing over the weights therefore gives the negative of the best
certificate. The norm constraints are second-order cones, and a rotated
second-order cone represents the quadratic objective's epigraph.
\end{proof}

\subsection{Proof of record-only minimax sharpness}
\label{app:minimax-proof}
We prove Theorem~\ref{thm:baseline-minimax} for $k<d$, allowing the
more general residual sets $K(c^{(i)})$ with
$v^{(i)}\in\R^k$ and $c^{(i)}\in\R_+^{k+1}$.
For simplex weights $\alpha$, write
$\beta=\sum_i\alpha_i b_i$, $v=\sum_i\alpha_i v^{(i)}$, and
$c=\sum_i\alpha_i c^{(i)}$, and set
$L_Q(\alpha)=\beta-\Mnorm(v,c)^2/(2\lambda)$.
The raw records in the theorem correspond to $v^{(i)}=v_i$ and
$c^{(i)}=\eta_i e_{m_i}$, where $e_{m_i}$ is the $m_i$th coordinate
vector of $\R^{k+1}$, indexed from zero. Then
$K(c^{(i)})=\eta_i\mathcal B_{m_i}$ and
$(\beta,v,c)=(b^\alpha,v^\alpha,c^\alpha)$.
The proof has three steps. First minimize over
the residuals at a fixed ambient point. Then show that the convex tail lift
has the same value as this completion problem. Finally evaluate the lift's
dual. This route avoids interchanging an infimum and a supremum over
residuals. Such an interchange would generally be unjustified: the norm
penalty in the certificate is convex in the residual, and the completion
problem is not a convex minimization in the residuals and the point jointly.

\paragraph{Step 1: minimize over completions at a fixed point.}
Let $V$ denote the smallest optimum among independent completions of
these records, as on the left side of \eqref{eq:baseline-minimax-equality}.
The two
infima, over $x$ and over the residual completions, commute. For any fixed
$x$, the residual sets are compact. Each cut independently attains its
smallest inner product, and all these minimizers can be chosen at once.
Consequently,
\begin{align}
 \inf_{r_i\in K(c^{(i)})}\max_i\{b_i+\ip{Qv^{(i)}+r_i}{x}\}
 &=\max_i\left\{b_i+\ip{Qv^{(i)}}{x}+\min_{r_i\in K(c^{(i)})}\ip{r_i}{x}\right\}\nonumber\\
 &=\max_i\left\{b_i+\ip{Qv^{(i)}}{x}-\sum_{j=0}^k c^{(i)}_j\norm{\Pi_{S_j^\perp}x}\right\}.
 \label{eq:independent-completion}
\end{align}
The last equality uses the minimum inner product over each orthogonal ball:
it is the negative of the radius times the corresponding projected norm.
Consequently,
\begin{equation}
 V=\inf_{x\in\R^d}\left\{\frac\lambda2\norm{x}^2+
 \max_i\left[b_i+\ip{Qv^{(i)}}{x}
 -\sum_{j=0}^k c^{(i)}_j\norm{\Pi_{S_j^\perp}x}\right]\right\}.
 \label{eq:completion-envelope}
\end{equation}
The resulting objective is not
yet convex in $x$, because the negative tail norms occur inside the maximum.

\paragraph{Step 2: prove that the tail lift is exact.}
Introduce the convex problem
\begin{equation}
\begin{aligned}
 \min_{z,t,s}\quad &\frac\lambda2 t_0^2+s\\
 \text{s.t.}\quad &b_i+(v^{(i)})^\top z-\sum_{j=0}^k c^{(i)}_jt_j\le s
                  \quad(1\le i\le N),\\
 &\norm{(z_j,t_j)}\le t_{j-1}\quad(1\le j\le k),\\
 &t_k\ge0,
\end{aligned}
 \label{eq:tail-lift}
\end{equation}
where $z\in\R^k$, $t\in\R^{k+1}$, and $s\in\R$.
Let $P$ be its optimum. We prove both inequalities
between $P$ and $V$. Any $x\in\R^d$ induces the feasible choice
\[
 z=Q^\top x,\qquad t_j=\norm{\Pi_{S_j^\perp}x},\qquad
 s=\max_i\left\{b_i+(v^{(i)})^\top z-\sum_jc^{(i)}_jt_j\right\}.
\]
Here $t_0=\norm{x}$ because $S_0=\{0\}$. The tail constraints hold with
equality by orthogonality, and the lifted objective equals the value in
\eqref{eq:completion-envelope}. Thus $P\le V$.

Conversely, take any feasible $(z,t,s)$ in the lift. The cone inequalities imply that every $t_j\ge0$ and, by repeated squaring,
\begin{equation}
 t_j^2\le t_0^2-\sum_{\ell=1}^j z_\ell^2
 \quad(0\le j\le k).
 \label{eq:tail-inequality}
\end{equation}
The lift may underestimate some tail lengths. The unused ambient direction
lets us fill the remaining length without changing the retained coordinates.
Since $k<d$, choose $w\in S_k^\perp$ with
$\norm{w}^2=t_0^2-\norm{z}^2$ and set $x=Qz+w$. The squared length is
nonnegative by \eqref{eq:tail-inequality}. This gives
\begin{equation}
 \norm{x}=t_0,\qquad
 \norm{\Pi_{S_j^\perp}x}
   =\sqrt{t_0^2-\sum_{\ell=1}^jz_\ell^2}\ge t_j.
 \label{eq:tail-completion}
\end{equation}
The coefficients $c^{(i)}_j$ are nonnegative, and the tail terms enter each
cut with a minus sign. Replacing $t_j$ by these larger actual tails cannot
increase any cut expression. The value of \eqref{eq:completion-envelope}
at $x$ is therefore at most $\lambda t_0^2/2+s$. Taking infima yields
$V\le P$, so the lift is exact. For $k=0$, the same argument simply chooses
a vector of norm $t_0$ in $\R^d$.

\paragraph{Step 3: evaluate the convex dual.}
The lift is convex and satisfies Slater's condition. A strict feasible point
is obtained by choosing $z=0$, strictly decreasing positive
$t_0,\ldots,t_k$, and sufficiently large $s$. The objective is also bounded
below. To see this, let $V_{\max}=\max_i\norm{v^{(i)}}$ bound the retained
slope norms and $C_{\max}=\max_i\sum_jc^{(i)}_j$ bound the total profile
radii. The inequalities $\norm{z}\le t_0$ and $t_j\le t_0$ imply
\[
 \frac\lambda2t_0^2+s\ge\frac\lambda2t_0^2+
 \min_i b_i-(V_{\max}+C_{\max})t_0.
\]
The positive quadratic term makes this lower bound finite. Slater's
condition therefore gives strong Lagrangian duality. Assign nonnegative
multipliers $\alpha_i$ to the affine cut inequalities. The coefficient of
the free epigraph variable $s$ in the Lagrangian is
$1-\sum_i\alpha_i$, so minimizing over $s$ enforces
$\sum_i\alpha_i=1$. For the resulting aggregate $(\beta,v,c)$, the dual
function is
\begin{equation}
 \beta+\inf_{(z,t)\in\mathcal T}
       \left\{\frac\lambda2t_0^2+v^\top z-\sum_jc_jt_j\right\},
 \label{eq:lift-dual}
\end{equation}
where $\mathcal T$ is the set of $(z,t)$ satisfying the tail cone
inequalities in \eqref{eq:tail-lift}, including $t_k\ge0$.

It remains to evaluate the inner infimum. Fix the full length $t_0=a\ge0$.
The smallest remaining linear term is $-a\Mnorm(v,c)$. Indeed,
\eqref{eq:tail-completion} replaces a feasible tail vector by actual tails
without increasing this term, while every ambient vector supplies feasible
tail variables. Thus we may minimize over ambient vectors of norm $a$.
Replacing $x$ by $-x$ reverses the retained linear term without changing any
tail norm; sign symmetry and \eqref{eq:norm-support} then give
\[
 \min_{\norm{x}=a}\left\{\ip{Qv}{x}-\sum_jc_j\norm{\Pi_{S_j^\perp}x}\right\}
 =-a\Mnorm(v,c).
\]
The support maximum in \eqref{eq:norm-support} can be taken on the unit
sphere: its objective is positively homogeneous and has nonnegative
maximum. Scaling by $a$ gives the displayed identity, including $a=0$.
The remaining minimization in \eqref{eq:lift-dual} is therefore the scalar
quadratic problem
\[
 \inf_{a\ge0}\left\{\frac\lambda2a^2-a\Mnorm(v,c)\right\}
 =-\frac{\Mnorm(v,c)^2}{2\lambda}.
\]
The dual optimum is exactly $\max_{\alpha\in\Delta_N}L_Q(\alpha)$, proving the equality.

\paragraph{Attainment and a worst compatible objective.}
The equality also has a compatible objective that attains it. First, the
lift attains its optimum: the quadratic lower bound above bounds $t_0$ on
every nonempty objective sublevel set. All other tails and $z$ are then
bounded, and the cut inequalities and objective bound together bound $s$.
The feasible set is closed, so these sublevel sets are compact.

Complete an optimal lift to $x$ using \eqref{eq:tail-completion}. For each
record choose residual components opposite to the tails of this $x$:
\begin{equation}
 r_i=-\sum_{j=0}^k c^{(i)}_j
       \frac{\Pi_{S_j^\perp}x}{\norm{\Pi_{S_j^\perp}x}},
 \label{eq:adversarial-residual}
\end{equation}
Set a term with zero denominator to zero. Each summand belongs to its
designated residual ball, so $r_i\in K(c^{(i)})$, and its inner product
with $x$ attains the minimum in \eqref{eq:independent-completion}. The
completed max-of-affines objective has value at $x$ at most $V$. By
definition of $V$, its minimum is at least $V$, so equality follows. Any
other compatible convex $h$ dominates some completed envelope, and hence
cannot have a smaller optimum than $V$. This proves the information claim
as well as the minimax equality.

\subsection{Counterexamples and exact residual losses}
\label{app:counterexamples}
\paragraph{Full-rank information can be stronger than an aggregate certificate.}
Take $d=k=1$ and $\lambda=1$, with an uncertain cut $1+rx$, $|r|\le1$,
and two exact cuts $x$ and $-x$. The uncertain slope was discarded at
dimension zero. Let $a\in[0,1]$ be the weight on this cut.
The remaining weights sum to $1-a$, and their signed
aggregate slope is the scalar $v$. The aggregate intercept is $a$, and at
full rank its profile is $c_0=a$, $c_1=0$. The recursion gives
$\Mnorm=|v|+a$, so
\[
 L_Q=a-\tfrac12(|v|+a)^2\le a-\tfrac12a^2\le\tfrac12.
\]
The value $1/2$ is attained at $a=1$. To compute the smallest compatible
optimum directly, put $t=|x|$. At each fixed $x$, the uncertain residual can
be chosen so that $rx=-|x|$, while the maximum of the two exact cuts is
$|x|$. Independently completing the uncertain cut and minimizing over $x$
therefore gives
\[
 \inf_{|r|\le1}\inf_x\left\{\tfrac12x^2+\max(1+rx,x,-x)\right\}
 =\min_{t\ge0}\left\{\tfrac12t^2+\max(1-t,t)\right\}.
\]
For $t\le1/2$, the expression is $t^2/2+1-t$, decreasing to $5/8$. For
$t\ge1/2$, it is $t^2/2+t$, increasing from $5/8$. This proves the strict
gap between the two bounds. In the lift proof, the missing direction was
needed to increase actual tail lengths at a fixed full norm; this example
shows that the step cannot be omitted at full rank.

\paragraph{An old residual overlaps a later discovered direction.}
In Example~\ref{ex:history}, write $y=x_1$. The objective is minimized
at $x_2=0$, and the two affine functions intersect at $y_*=-(1+\eta)/2$.
The derivative of the quadratic-plus-first-cut branch
there is $(\eta-1)/2<0$, and that of the second branch is
$(1-\eta)/2>0$, so $y_*$ is the minimizer. Evaluation yields
$F_*=(\eta^2-2\eta-3)/8$. The first cut is active at $y=-1$ and the second
at $y=0$, so both can be genuine oracle minorants.

Their equally weighted intercept is $(\eta^2-1)/4$. The correct norm
$(1+\eta)/2$ gives exactly $F_*$. Treating the old residual as orthogonal
to the retained slope instead uses the fictitious norm
$\sqrt{1+\eta^2}/2$ and gives $(\eta^2-3)/8$, proving the error $\eta/4$.

\subsection{Further consequences of the stored geometry}
\label{app:consequences}
\paragraph{Deleting cuts does not erase an already proved bound.}
Each evaluated scalar $L_Q$ remains a lower bound for the same target
objective after its source cuts are deleted. Keeping their maximum therefore
gives monotone $\Lbest$. Replacing a selected aggregate by one record also
preserves its residual profile and all subsequent nonnegative uses of that
record. It does not preserve the maximum envelope of the original cuts.

For a concrete example, set $\lambda=1$ and begin with the exact cuts $x$
and $-x-4$. Their regularized model minimum is $-1/2$, supported by $x$
alone, so keeping only that cut preserves the current certificate. Now add
$3x+7$. Keeping only $x$ from the old model gives minimum $5/2$ at $x=-3$.
Keeping both old cuts gives minimum $81/32$ at their relevant intersection
$x=-11/4$. Thus a cut with zero weight in the present certificate can
improve a future certificate.

\paragraph{Basis truncation.}
The append-only formulas do not directly apply when basis columns are
removed. To truncate a raw record born at $S_m$ to a prefix
$S_\ell\subseteq S_m$, project its retained slope to
$p'=\Pi_{S_\ell}p$. The new omitted vector is the dropped known component
$p-p'$ plus the old residual. A valid new radius is
\[
 \eta'=\sqrt{\norm{p-p'}^2+\eta^2}.
\]
The dropped known vector lies in $S_m$, the old residual is orthogonal to
$S_m$, and their sum is orthogonal to $S_\ell$; Pythagoras proves the
radius. The updated record is labeled with birth dimension $\ell$.
Records born before $S_\ell$ keep their earlier birth labels.

For an arbitrary change of basis, a safe update can instead use a translated
uncertainty set. If $p'$ is the new retained slope and $K$ the old residual
set, take $K'=(p-p')+K$, since
$p+r=p'+[(p-p')+r]$. Merely reusing the old radius or relabeling its
orthogonality is unjustified.

\paragraph{Distance to the target minimizer.}
For the quadratic target, an objective-gap certificate also controls
distance to its minimizer. If a minimizer $x_*$ exists, convexity of $h$ and
the quadratic term imply
$F(x)-F_*\ge\lambda\norm{x-x_*}^2/2$. Therefore any reported gap
$\Delta=U-\Lbest$ gives
$\norm{\widehat x-x_*}\le\sqrt{2\Delta/\lambda}$. This statement concerns
the minimizer of the regularized objective; it does not certify a separate
discrete assignment or primal feasibility for another problem.

\section{General domains, oracle errors, and protocol certificates}
\label{app:extensions}
Theorem~\ref{thm:baseline-robust} has two ingredients: sum valid minorants with
nonnegative coefficients, and bound the conjugate penalty for their
aggregate uncertainty. This appendix uses the same argument with other
domains and regularizers, more general oracle records, and outputs formed
without function values. These are certificate constructions; they do not
give convergence rates for arbitrary proposal methods.
For a fixed weight vector, we use the abbreviations
$\beta=b^\alpha$, $p=p^\alpha$, and $K_\alpha=K^\alpha$.

\subsection{Balls and other known regularizers}
\label{app:ball-regularizer}
\paragraph{Conjugate penalties on a ball.}
First take the known domain to be the radius-$R$ Euclidean ball. For
$\Psi=I_{R\Ball^d}$, maximizing a linear function over this ball gives
the conjugate $R\norm{y}$. Thus the exact bound for fixed aggregate
$(\beta,v,c)$ is $L_R=\beta-R\Mnorm(v,c)$. If the target also has a
quadratic regularizer with $\lambda>0$, the conjugate changes to
\begin{equation}
 \Psi(x)=\frac\lambda2\norm{x}^2+I_{R\Ball^d}(x),\qquad
 \Psi^*(y)=q_R(\norm{y}),\qquad
 q_R(s)=\begin{cases}s^2/(2\lambda),&s\le\lambda R,\\Rs-\lambda R^2/2,&s>\lambda R.
 \end{cases}
 \label{eq:regularized-ball-conjugate}
\end{equation}
Here $s=\norm{y}\ge0$, and $q_R$ is the scalar conjugate penalty as a
function of slope norm. To derive it, align $x$ with $y$ and write
$t=\norm{x}$. Conjugacy reduces to maximizing $st-\lambda t^2/2$ over
$0\le t\le R$, whose maximizing radius is $\min\{s/\lambda,R\}$.
Since $q_R$ is nondecreasing,
$\sup_{r\in K(c)}\Psi^*(-Qv-r)=q_R(\Mnorm(v,c))$.

The record-only minimax equality extends to both objectives on the ball
when $k<d$ and $R>0$. Add $t_0\le R$ to the tail lift, because $t_0$
represents the ambient norm. Its exactness proof is unchanged, and a strict
feasible point has $0<t_k<\cdots<t_0<R$. The dual inner minimization is
now restricted to radii in the ball:
$\inf_{0\le a\le R}\{\lambda a^2/2-a\Mnorm\}=-q_R(\Mnorm)$.
For the unregularized case, set $\lambda=0$ in this scalar minimization
to obtain $-R\Mnorm$. This proves both extensions.

\paragraph{Residual loss and artificial curvature.}
For common-space records, write $\rho=\sum_i\alpha_i\eta_i$ for the
weighted residual radius. The unregularized-ball correction, namely the
decrease from the retained-slope bound to the ambient bound, is
\[
 R\left(\sqrt{\norm{v}^2+\rho^2}-\norm{v}\right)\le R\rho.
\]
At $v=0$, an omitted affine slope of norm $\rho$ attains the loss $R\rho$,
so its worst-case linear dependence is unavoidable.

Adding artificial curvature does not remove this issue for an originally
unregularized objective $f$. Denote its optimum by $f_*$ and the optimum
after adding a quadratic by $F_{\lambda,*}$. If some original minimizer has
norm at most $R$, then
\[
 f(x)-f_*\le F_\lambda(x)-F_{\lambda,*}+\lambda R^2/2,
 \qquad F_\lambda=f+\lambda\norm{x}^2/2.
\]
This follows from $F_{\lambda,*}\le f_*+\lambda R^2/2$ and
$f(x)\le F_\lambda(x)$. The extra term $\lambda R^2/2$ accounts for the
change of target caused by regularization. Balancing this bias against the
common-space residual cost gives
\[
 \inf_{\lambda>0}\left\{\lambda R^2/2+\rho^2/(2\lambda)\right\}=R\rho.
\]
Thus the quadratic residual allowance is most useful when curvature already belongs to the target.

\paragraph{Other radial regularizers.}
More generally, let $\Psi(x)=\phi(\norm{x})$, where $\phi$ is
nondecreasing and convex on nonnegative radii. Define the scalar conjugate
over radii by $\phi_+^*(s)=\sup_{t\ge0}\{st-\phi(t)\}$. The same radial
alignment argument gives the conjugate $\phi_+^*(\norm{y})$, and
monotonicity gives the fixed-aggregate bound
$\beta-\phi_+^*(\Mnorm)$. Useful finiteness of this penalty remains an
explicit condition.

\subsection{General strong convexity and a fixed objective metric}
Radial symmetry makes the historical penalty exact through a single norm.
For a general strongly convex regularizer, we can instead bound the change
in the conjugate around the retained slope. Suppose $\Psi$ is
$\mu$-strongly convex in a norm $\lVert\cdot\rVert$. Let
$\lVert\cdot\rVert_*$ denote its dual norm, and suppose
$\dnorm{r}\le\rho$ for all $r\in K_\alpha$. Define
$x_p=\nabla\Psi^*(-p)$, the minimizer for the retained affine slope, and
$L_0=\beta-\Psi^*(-p)$, the corresponding uncorrected value.
Strong convexity and $-p\in\partial\Psi(x_p)$ imply
\[
 \Psi(x_p+w)\ge\Psi(x_p)-\ip{p}{w}+\frac\mu2\lVert w\rVert^2.
\]
Subtract this inequality from $\ip{-p-r}{x_p+w}$ and take the supremum
over $w$. H\"older's inequality bounds the remaining term by
$\sup_{t\ge0}\{\dnorm{r}t-\mu t^2/2\}=\dnorm{r}^2/(2\mu)$.
To express the resulting linear penalty, define the support function of a
set $K$ by $\sigma_K(y)=\sup_{r\in K}\ip{y}{r}$. Consequently,
\[
 \Psi^*(-p-r)\le\Psi^*(-p)-\ip{x_p}{r}+\dnorm{r}^2/(2\mu),
 \qquad
 L(\alpha)\ge L_0-\sigma_{K_\alpha}(-x_p)-\rho^2/(2\mu).
\]
The support-function term accounts for first-order interaction between the
omitted slope and $x_p$; the last term is quadratic in the residual radius.
For nested Euclidean balls the linear term is
$\sum_jc_j\norm{\Pi_{S_j^\perp}x_p}$. It vanishes only with the
appropriate orthogonality. For example,
$\Psi(x)=\mu\norm{x}^2/2+\ip{q}{x}$, with a fixed vector $q$, has
conjugate $\norm{y-q}^2/(2\mu)$. An omitted slope aligned with $q$ changes
the conjugate to first order even when the retained slope is zero.

If the target has a fixed positive-definite quadratic metric $H\succ0$,
an exact Euclidean treatment is available by changing coordinates. Set
$y=H^{1/2}x$; then $x^\top Hx/2=\norm{y}^2/2$ and a slope $a_i$ becomes
$H^{-1/2}a_i$. Compress these transformed slopes in nested Euclidean spaces
and use $\beta-\Mnorm^2/2$. The metric must belong to the target
objective. An algorithmic quasi-Newton or proximal matrix does not
automatically supply such curvature.

\subsection{Joint errors, separators, and component information}
\paragraph{Correlated intercept and slope errors.}
If an intercept cannot be computed before discarding slope coordinates, a
residual norm alone may be insufficient. Retain a joint set $\mathcal U_i$
containing the actual pair $(s_i,r_i)$, where $s_i$ is the intercept error
and $r_i$ the omitted slope. The asserted minorant is then
$b_i+s_i+\ip{p_i+r_i}{x}$. Nonnegative summation gives the aggregate pair
set $\mathcal U_\alpha=\sum_i\alpha_i\mathcal U_i$, and the same
conjugacy proof gives
\begin{equation}
 L=\beta-\sup_{(s,r)\in\mathcal U_\alpha}\{\Psi^*(-p-r)-s\}.
 \label{eq:joint-error-bound}
\end{equation}
This preserves any known correlation between the intercept error and slope error.

\paragraph{Feasibility separators.}
A feasibility separator is an affine function known to be nonpositive on
$X$. If $b_j+\ip{p_j+r_j}{x}\le0$ on $X$, it can be added to an objective
minorant with any multiplier $\gamma_j\ge0$, since the added term is
nonpositive at every feasible point. Objective-cut weights still sum to
one, but separator weights need not. Both the separator intercepts and
their residual sets are scaled by $\gamma_j$.

\paragraph{Partial objective information.}
A \emph{mass record} bounds a scalar multiple of the loss:
$\tau_i h(x)\ge b_i+\ip{p_i+r_i}{x}$. To obtain a minorant of $h$, use
nonnegative coefficients satisfying $\sum_i\alpha_i\tau_i=1$.
For a component objective $h=\sum_{\ell=1}^J w_\ell h_\ell$, let
$w=(w_1,\ldots,w_J)$ be the vector of component weights. A vector-mass
record has mass vector $\tau_i=(\tau_{i1},\ldots,\tau_{iJ})$ and asserts
$\sum_\ell\tau_{i\ell}h_\ell(x)\ge b_i+\ip{p_i+r_i}{x}$.
Choose nonnegative coefficients satisfying $\sum_i\alpha_i\tau_i=w$.
In each case, summing the valid inequalities recovers a minorant for the
target loss, to which the same conjugate bound applies. These formulations
allow componentwise or distributed minorants without a full objective
gradient at every call.

\paragraph{Probabilistic records.}
Let $\delta_{\mathrm{prob}}$ be a failure-probability allowance. If all
minorant and residual-containment assertions hold simultaneously on an
event of probability at least $1-\delta_{\mathrm{prob}}$, the pathwise
proof holds at every adaptive stopping time on that event. Unbiasedness of
a stochastic gradient does not establish this simultaneous containment
condition.

\subsection{Inexact full-objective subgradients and usable curvature}
\label{app:inexact-curvature}
For an exact full-objective subgradient of a quadratic target, subtracting
the regularizer gradient gives a loss subgradient. For an inexact
full-objective subgradient, the same subtraction requires an adjusted error
allowance and may require less curvature. Let $F$ be $\lambda$-strongly
convex and let $g\in\partial_\delta F(x)$ in the ordinary global sense.
Choose $0<\theta<1$ as the fraction of curvature to retain. Then
\begin{equation}
 F(y)\ge F(x)+\ip{g}{y-x}+\frac{\theta\lambda}{2}\norm{y-x}^2
                              -\frac\delta{1-\theta}.
 \label{eq:inexact-curvature}
\end{equation}
To prove this inequality, set $t=1-\theta$ and take the intermediate point
$z=(1-t)x+ty$. Applying the inexact subgradient inequality at $x$ to this
$z$, and bounding $F(z)$ by strong-convexity interpolation, gives,
respectively,
\[
 F(z)\ge F(x)+t\ip{g}{y-x}-\delta,
 \qquad
 F(z)\le(1-t)F(x)+tF(y)-\frac\lambda2t(1-t)\norm{y-x}^2.
\]
Combining these bounds and dividing by $t$ proves
\eqref{eq:inexact-curvature}. Subtracting
$\theta\lambda\norm{y}^2/2$ from both sides shows that
$g-\theta\lambda x$ is a $\delta/(1-\theta)$-subgradient of
$F-\theta\lambda\norm{x}^2/2$. This supplies a valid minorant for a
decomposition using the known curvature $\theta\lambda$.

The choice of $\theta$ trades retained curvature against the oracle-error
allowance. For example, half the curvature doubles the error allowance.
If $\delta=0$, taking $\theta\uparrow1$ recovers exact full-curvature
subtraction. For an arbitrary inexact full-objective subgradient, simply
subtracting $\lambda x$ need not be valid.

\subsection{Value-free averaged outputs and variational inequalities}
\label{app:value-free}
Return to the general target \eqref{eq:baseline-problem}, with
$\Psi=\psi_0+I_B$. Function values can be eliminated if we certify an
average of the query points rather than an incumbent selected by objective
value. Suppose $a_i\in\partial_{\delta_i}h(x_i)$ at feasible points, and
store the exact scalar $s_i=\ip{a_i}{x_i}$ before compression. Write the
compressed slopes as $a_i=p_i+r_i$, and form $p$ and $K_\alpha$ as in
Theorem~\ref{thm:baseline-robust}.

For $\alpha\in\Delta_N$ and $\bar x=\sum_i\alpha_i x_i$, Jensen's
inequality gives $h(\bar x)\le\sum_i\alpha_i h(x_i)$. Summing the
inexact minorants gives, for every $z\in X$,
\[
 F(z)\ge\sum_i\alpha_i h(x_i)-\sum_i\alpha_i(s_i+\delta_i)
          +\ip{p+r_{\mathrm{true}}}{z}+\psi_0(z).
\]
Here $r_{\mathrm{true}}=\sum_i\alpha_i r_i$. Minimize over $z\in X$,
enlarge to $B$, and use the conjugate to obtain a lower bound. Subtracting
it from the Jensen upper bound for $F(\bar x)$ cancels the unknown
function values, yielding
\begin{equation}
 F(\bar x)-F_*\le\psi_0(\bar x)+\sum_i\alpha_i(s_i+\delta_i)
                       +\sup_{r\in K_\alpha}\Psi^*(-p-r).
 \label{eq:value-free}
\end{equation}
On the unregularized ball this becomes
$\sum_i\alpha_i(s_i+\delta_i)+R\Mnorm$. For the unconstrained quadratic,
it becomes
$\lambda\norm{\bar x}^2/2+\sum_i\alpha_i(s_i+\delta_i)+\Mnorm^2/(2\lambda)$.
Feasibility separators with multiplier $\gamma_j$ add $-\gamma_jb_j$ to
the scalar term and add their slopes and residuals to the aggregate. Their
nonpositivity on $X$ proves the same bound.

Selecting weights after the fact requires enough information about the
query points to construct $\bar x$. When the queries lie in the discovered
span, stored coordinates suffice. Arbitrary ambient proposal methods must
keep the points or a suitable primal aggregate. The incumbent certificate
$U-L$ avoids this requirement and needs only one best feasible output.

\paragraph{Monotone-operator certificates.}
The same summation argument also bounds a variational-inequality residual.
Let $V$ be a monotone operator on $X$, meaning
$\ip{V(x)-V(y)}{x-y}\ge0$ for all $x,y\in X$. Store
$s_i=\ip{V(x_i)}{x_i}$ and a valid compressed representation of $V(x_i)$.
Monotonicity implies
$\ip{V(y)}{x_i-y}\le\ip{V(x_i)}{x_i-y}$. Summing with simplex weights
and maximizing over $y\in X$ gives the Minty residual bound
\begin{equation}
 \sup_{y\in X}\ip{V(y)}{\bar x-y}
 \le\sum_i\alpha_i s_i+
                  \sup_{r\in K_\alpha}\sigma_X(-p-r).
 \label{eq:vi-bound}
\end{equation}
The left side measures the Minty residual at the averaged point $\bar x$;
on the right, $\sigma_X$ is the support function of the feasible domain.
The same reasoning applies to the usual ergodic residual for a
convex--concave saddle operator. These are support-function certificates;
the objective's quadratic correction does not automatically extend to
arbitrary monotone games.

\paragraph{Numerical orthogonality.}
All historical formulas assume that the actual omitted vector belongs to
the stated orthogonal ball. In finite precision, an approximately
orthonormal $Q$ and a measured residual norm do not alone prove this exact
membership. One can certify orthogonality and outward norm bounds, or
enclose numerical leakage parallel to the stored space in an additional
full-ball component at birth dimension zero. That component does not assert
orthogonality to any retained direction. If only a full norm bound is
trustworthy, use that uncertainty set in Theorem~\ref{thm:baseline-robust}.
This preserves validity at the price of a weaker certificate.

\section{The bounded-ball cutting-plane construction}
\label{app:cutting-plane}
\subsection{Validity of the history-aware, value-free certificate}
\label{app:cp-validity}
We first prove that each choice of weights certifies its own averaged
output, independently of the volume argument. Let the retained productive
points satisfy \eqref{eq:ball-oracle}, and choose simplex weights $\alpha$
on these records. Their omitted slopes are $r_i=g_i-p_i$, orthogonal to
the space containing $x_i$. This remains true if the basis is enlarged
before the record is stored: $x_i$ still belongs to that enlarged space.
Thus $\ip{r_i}{x_i}=0$.

For $\bar x^\alpha=\sum_i\alpha_i x_i$ and every $z\in X$, convexity and the
oracle inequalities give
\begin{align}
 F(\bar x^\alpha)-F(z)
 &\le\sum_i\alpha_i\{F(x_i)-F(z)\}\\
 &\le\sum_i\alpha_i\{\ip{g_i}{x_i-z}+\delta_i\}\\
 &=\sum_i\alpha_i(\ip{p_i}{x_i}+\delta_i)
                      -\ip{\sum_i\alpha_i(p_i+r_i)}{z}.
\end{align}
The first inequality is Jensen's inequality. In the last equality, the
residuals disappear from the query-point scalar products because
$\ip{r_i}{x_i}=0$. Maximizing the remaining linear term over the ball
yields $R\norm{\sum_i\alpha_i(p_i+r_i)}$. The actual aggregate residual
belongs to $K(c^\alpha)$, so Lemma~\ref{lem:baseline-norm} proves
Lemma~\ref{lem:ball-certificate}. For the volume argument, we also use the
simpler triangle certificate
\begin{equation}
 C_{\mathrm{tri}}(\alpha):=
 \sum_i\alpha_i(\ip{p_i}{x_i}+\delta_i+R\eta_i)
       +R\norm{\sum_i\alpha_i p_i}.
 \label{eq:triangle-certificate}
\end{equation}
Indeed, the triangle inequality gives
\[
 \Mnorm(v^\alpha,c^\alpha)
 \le\norm{\sum_i\alpha_i p_i}+\sum_i\alpha_i\eta_i,
\]
so $C(\alpha)\le C_{\mathrm{tri}}(\alpha)$. No optimality
condition on the weights is needed for either certificate. If $p_i=0$, a
single record gives $F(x_i)-F_*\le\delta_i+R\eta_i$ and the algorithm can
stop immediately.

\subsection{The volume primitive and reserved-coordinate symmetry}
\label{app:primitive}
Assumption~\ref{ass:primitive} isolates the fixed-dimensional geometric
result used in Theorem~\ref{thm:cp}. A volumetric localizer starts from a
box, accepts a nonzero central halfspace at its query center, retains a
permitted outward relaxation, and may delete constraints with small
leverage scores. It retains $O(q)$ constraints and, after
$O(q\log(2+qR/r))$ updates, its volume is below that of a radius-$r$ ball.
The fixed-dimensional reduction and deletion analysis are prior ingredients
\citep{anstreicher1997}; Appendix A of \citet{jiang2020} describes the
volumetric barrier, leverage deletion, outward cut insertion, and $O(q)$
constraint count. This interface is specific to a suitable volumetric
primitive; arbitrary analytic-center methods or deletion policies need not
provide it.

We verify why coordinates can remain unassigned until a direction is
discovered. Within a phase, let $m$ be the current number of assigned
coordinates and $q$ the fixed capacity. Every newly supplied cut normal
has zero entries in the unassigned coordinates. The polytope therefore has
product form
\begin{equation}
 \mathcal P=\mathcal P_{\mathrm{active}}\times[-R,R]^{q-m}.
 \label{eq:product-polytope}
\end{equation}
Here $\mathcal P_{\mathrm{active}}$ is the polytope in the assigned
coordinates. To describe its volumetric center and deletion rule, let $a_i$
be the polytope's facet normals and $s_i(y)>0$ the slacks of their defining
inequalities at an interior point $y$. Define
\[
 H(y)=\sum_i\frac{a_ia_i^\top}{s_i(y)^2},\qquad
 V_{\mathrm{vol}}(y)=\frac12\log\det H(y),\qquad
 \ell_i(y)=\frac{a_i^\top H(y)^{-1}a_i}{s_i(y)^2}.
\]
The matrix $H$ is the Hessian of the logarithmic barrier used to define the
volumetric barrier $V_{\mathrm{vol}}$. The leverage score $\ell_i$
measures a facet's contribution to $H$ and is used in the deletion rule.
For \eqref{eq:product-polytope}, $H$ is block diagonal. An unassigned
coordinate $z_j$ contributes the scalar block
\[
 H_j(z_j)=\frac1{(R-z_j)^2}+\frac1{(R+z_j)^2}.
\]
This even function is minimized uniquely at zero: its derivative is
positive for $z_j>0$ and negative for $z_j<0$. The determinant of a block
diagonal matrix is the product of its block determinants, so the barrier
splits into an active term and $\frac12\sum_j\log H_j(z_j)$. Every
unassigned coordinate of its center is therefore zero.

Newton centering initialized with these coordinates at zero preserves them:
the gradient vanishes there and the centering system is block diagonal.
Approximate centering can likewise enforce these zeros exactly in working
coordinates.

At $z_j=0$, each of its two box facets has leverage
$R^{-2}/(2R^{-2})=1/2$. A standard deletion threshold strictly below $1/2$
therefore retains both facets, without an exceptional deletion rule. Outward
relaxation of a supplied cut with zero unused entries cannot break the
product structure. Assigning a discovered direction to an unused coordinate
changes no old cut coefficient and no current point, because that coordinate
was zero.

These observations establish the symmetry and facet-protection part of the
interface. The rate still relies on the chosen primitive's fixed-dimensional
volume theorem; it is not a new arithmetic or bit-complexity theorem.

\subsection{Details of Algorithm~\ref{alg:adaptive-cp}}
\label{app:cp-algorithm}
Algorithm~\ref{alg:adaptive-cp-detailed} expands the steps in the main
text. Let $P$ index the retained objective records and $\Delta_P$ be
their simplex. The center $y\in\R^q$ has zero unassigned entries;
the ambient query is $x_i=Qy_{1:m}$. For a retained slope $p_i$,
$\widetilde p_i\in\R^q$ contains its basis coordinates, padded with
zeros in unassigned slots.

\begin{algorithm}[htbp]
\caption{Detailed dimension-adaptive cutting-plane method}
\label{alg:adaptive-cp-detailed}
\small
\begin{algorithmic}[1]
\Require $R,\delta,G$, discovery threshold $\tau$, tolerance $\zeta$
\State $Q\gets[\,]$; $m\gets0$; $q\gets1$; $P\gets\varnothing$; $i\gets0$
\State $\mathrm{CP.Init}([-R,R]^q)$
\Loop
  \State $y\gets\mathrm{CP.Center}()$; prune $P$ to retained cuts
  \If{$P\ne\varnothing$}
    \State $\alpha\gets\argmin_{\alpha\in\Delta_P}C(\alpha)$
    \If{$C(\alpha)\le\delta+R\tau+\zeta$}
      \State \Return $(\bar x^\alpha,C(\alpha))$
    \EndIf
  \EndIf
  \If{$\norm{y}\ge R$}
    \State $\mathrm{CP.Cut}(y,y)$; \textbf{continue}
  \EndIf
  \State $i\gets i+1$; $x_i\gets Qy_{1:m}$
  \State Query $(g_i,\delta_i)$; $r_i\gets g_i-QQ^\top g_i$
  \If{$\norm{r_i}>\tau$}
    \State $Q\gets[Q\ \ r_i/\norm{r_i}]$; $m\gets m+1$
    \State $p_i\gets g_i$; $\eta_i\gets0$
  \Else
    \State $p_i\gets QQ^\top g_i$; $\eta_i\gets\norm{r_i}$
  \EndIf
  \If{$m>q$}
    \State $q\gets\min\{2q,d\}$; $P\gets\varnothing$
    \State $\mathrm{CP.Init}([-R,R]^q)$; \textbf{continue}
  \EndIf
  \If{$p_i=0$}
    \State \Return $(x_i,\delta_i+R\eta_i)$
  \EndIf
  \State $m_i\gets m$; $v_i\gets Q^\top p_i$; $y_i\gets y$
  \State Store $(y_i,v_i,\eta_i,m_i,\delta_i)$ in $P$
  \State $\mathrm{CP.Cut}(\widetilde p_i,y)$, linked to record $i$
\EndLoop
\end{algorithmic}
\end{algorithm}

The basis remains append-only across phases. Within a phase, old slope
coordinates are padded with zeros as the basis grows, while recording
dimensions remain fixed. The primitive may replace a supplied central
cut by an outward relaxation. A newly assigned coordinate was zero at
the query, so the projected cut remains central.

The ball-separation cut contains the entire working ball: for
$\norm{y}\ge R$ and $\norm{z}\le R$,
$\ip{y}{z}\le R\norm{y}\le\norm{y}^2$.
It requires no objective call. An overflow query is not used as a cut
after restarting, since it was not taken at the new center.

Objective records are removed only when their associated constraints
are deleted. Weight selection is performed after this synchronization
and before the next objective call. One may also test the certificate
immediately after cut insertion; subsequent deletion cannot invalidate
a previously evaluated certificate. All stored query and slope
coordinates are in the phase's working space; reserving coordinates
requires no ambient directions to be chosen in advance.

\subsection{Proof of the oracle and storage bounds}
\label{app:cp-proof}
The proof has two parts. Within a phase, failure of the stopping test
forces a small ball to remain inside the localizer. The volume primitive
eventually excludes such a ball. We then sum the costs over the doubling
phases.

\paragraph{A projected certificate within one phase.}
Let $P$ be the index set of retained productive records, and let
$\Delta_P$ be the simplex of weights indexed by $P$. For $i\in P$, write
$y_i\in\R^q$ for the padded working coordinates of $x_i$ and
$\widetilde p_i\in\R^q$ for those of $p_i$. Because the discovered basis
is orthonormal, these coordinates preserve scalar products and norms. For
nonempty $P$, define the best projected certificate by
\begin{equation}
 D_P=\min_{\alpha\in\Delta_P}
 \left\{\sum_{i\in P}\alpha_i\ip{\widetilde p_i}{y_i}
                         +R\norm{\sum_{i\in P}\alpha_i\widetilde p_i}\right\}.
 \label{eq:phase-projected-residual}
\end{equation}
This certificate uses only retained slopes; the oracle errors and residual
radii are accounted for next. If $D_P\le\zeta$, use its minimizing weights
in the triangle certificate. The individual inequalities
$\delta_i+R\eta_i\le\delta+R\tau$ then give
\[
 \min_\alpha C(\alpha)
 \le C_{\mathrm{tri}}(\alpha)\le D_P+\delta+R\tau
 \le\delta+R\tau+\zeta.
\]
Thus a phase which has not stopped must have $D_P>\zeta$, unless $P$ is
empty. This is the only implication about weight selection needed for the volume
argument.

\paragraph{Failure to stop leaves a ball in the localizer.}
For nonempty $P$, bilinear minimax on the compact simplex and working ball
gives
\begin{equation}
 D_P=\max_{\norm{z}\le R}\min_{i\in P}\ip{\widetilde p_i}{y_i-z}.
 \label{eq:phase-minimax}
\end{equation}
If $D_P>\zeta$, there exists $z\in R\Ball^q$ with
$\ip{\widetilde p_i}{y_i-z}>\zeta$ for every retained productive cut.
This common margin allows a neighborhood of $z$ to satisfy all the cuts.
To keep the neighborhood within the working ball, let $r=\zeta/(4G)$
and move its center slightly toward zero: $w=(1-r/R)z$. Since
$0<\zeta\le RG$, we have $r\le R/4$, and the closed radius-$r$ ball
$B(w,r)$ lies in $R\Ball^q$. For every $u$ in this smaller ball,
$\norm{u-z}\le2r$. Since $\norm{\widetilde p_i}\le G$, it follows that
\[
 \ip{\widetilde p_i}{y_i-u}>\zeta-2Gr=\zeta/2>0.
\]
The small ball therefore obeys every productive central cut and every
outward relaxation of it. It also obeys each retained ball-separation cut
and box facet, because it lies in $R\Ball^q$. These are all the types of
constraints in the localizer, so the ball lies in the current polytope even
after old cuts have been deleted. If $P$ is empty, there are no productive
cuts to check and the entire working ball remains inside the localizer,
giving the same volume obstruction.

Assumption~\ref{ass:primitive} rules out this radius-$r$ ball after
$O(q\log(2+qRG/\zeta))$ updates. A phase must therefore stop or overflow
within that many updates. Ball-separation and cut-deletion updates require
no objective call, so counting all updates upper-bounds the productive
calls. Each overflow adds at most one extra objective query beyond that
phase's processed cuts: the query that discovered the overflowing direction.

\paragraph{Sum over phases and count storage.}
Let $m$ now denote the total number of directions discovered before
termination. If $m=0$, the first query has $p_i=0$, so the method stops
after one objective-oracle call and uses $O(d)$ storage. For $m\ge1$,
capacities double and their final value is at most $2m$. Their sum,
including a possible final value clipped at $d$, is at most $4m$. Hence
\[
 N_{\mathrm{oracle}}
 \le C\sum_{\mathrm{phases}}q\log(2+qRG/\zeta)+O(\log(m+1))
 =O\!\left(m\log(2+mRG/\zeta)\right).
\]
The basis needs $dm$ real numbers. Only the current phase's records need
to be retained. That phase has $O(q)$ constraints and productive records,
each with $O(q)$ slope, point, and profile coordinates. Dense centering and
weight-selection workspaces have dimension $O(q)$ and require $O(q^2)$ storage.
Including an ambient output, the resulting dense storage count is
$O(dm+m^2)$. This completes Theorem~\ref{thm:cp}.

\paragraph{Approximate weight-selection solves.}
Every feasible approximate weight vector safely certifies its own averaged
point. To preserve the termination rate, the solve also needs a controlled
suboptimality bound. For example, let $\kappa$ be an additive solve-error
allowance and require
$C(\alpha)\le\min C+\kappa$ with
$\kappa\le\zeta/2$. If $D_P\le\zeta/2$, the previous comparison ensures
that these approximate weights still pass the stopping threshold. If the
method has not stopped, the obstruction proof can therefore use margin
$\zeta/2$ and radius $\zeta/(8G)$, changing only constants in the
logarithm. Feasible weights alone establish validity but do not establish
this termination bound.

\subsection{Proof of the approximate-width discovery bound}
\label{app:width}
Each appended direction measures a component not in the span of earlier
discovery gradients. Let $a_1,\ldots,a_m$ be the gradients that trigger
discovery, listed in discovery order. The current basis spans their
predecessors, so
\[
 \dist(a_j,\spanop(a_1,\ldots,a_{j-1}))>\tau,
 \qquad \norm{a_j}\le G.
\]
Assume a fixed $k$-dimensional space $U$ satisfies
$\dist(a_j,U)\le\eta\le\tau/2$. If $\eta=0$, these gradients lie in
$U$ and are linearly independent, giving $m\le k$, also for exact
discovery with $\tau=0$. If $k=0$, no vector exceeds the threshold and
$m=0$. For the approximate bound take $\tau>0$, $k\ge1$, and $m>k$,
since $m\le k$ already implies the stated result.

The determinant argument compares two measures of the volume spanned by
the discovery gradients. Their sequential distances give a lower bound;
closeness to $U$ gives an upper bound.

For the matrix $A=[a_1\ \cdots\ a_m]$, Gram--Schmidt gives
\begin{equation}
 \det(A^\top A)=\prod_{j=1}^m
       \dist(a_j,\spanop(a_1,\ldots,a_{j-1}))^2>\tau^{2m}.
 \label{eq:gram-lower}
\end{equation}
Let $s_1\ge\cdots\ge s_m>0$ be the singular values of $A$. The slope
norm bound gives $\sum_js_j^2\le mG^2$. Meanwhile, $\Pi_UA$ has rank at
most $k$ and $\fnorm{A-\Pi_UA}^2\le m\eta^2$. The best rank-$k$
approximation has no larger Frobenius error, so
$\sum_{j>k}s_j^2\le m\eta^2$. The determinant is the product of all
$s_j^2$. Applying arithmetic--geometric mean separately to the first $k$
and remaining $m-k$ values therefore yields
\begin{equation}
 \tau^{2m}<\det(A^\top A)
 \le\left(\frac{mG^2}{k}\right)^k
       \left(\frac{m\eta^2}{m-k}\right)^{m-k}.
 \label{eq:gram-upper}
\end{equation}
To turn this comparison into a count, put $u=m/k>1$ and
$L=\log(G/\tau)\ge0$. Taking logarithms and using $\eta\le\tau/2$ gives
\[
 2(u-1)\log2<2L+\log u+(u-1)\log\frac{u}{u-1}
 \le2L+\log u+1.
\]
The last inequality follows from $\log(1+t)\le t$. Also $\log u\le u/2$ for $u>0$, as the minimum of $u/2-\log u$ occurs at $u=2$ and is positive. Thus
\[
 (2\log2-1/2)u<2L+1+2\log2,
 \qquad u<3(1+L).
\]
The second inequality follows because $3(2\log2-1/2)>1+2\log2$ and $3(2\log2-1/2)>2$. Since $m\le d$, this proves \eqref{eq:width-bound}.

\paragraph{Slope components already in the discovered span.}
The same proof applies when the discovery vectors are
$g_j=a_j+\xi_j$, where $\xi_j\in S_{j-1}$ is a vector in the previously
discovered span, $\norm{a_j}\le G_{\mathrm{new}}$, and
$\dist(a_j,U)\le\eta$. Projecting onto $S_{j-1}^\perp$ removes
$\xi_j$, and adjoining $g_j$ or $a_j$ gives the same span. Thus use
$G_{\mathrm{new}}$ for the discovery bound, while the full-gradient bound
$G$ still governs phase volume. In particular, the gradient $\lambda x_j$
of an isotropic regularizer lies in the query span and never triggers an
additional direction.

\paragraph{Why equal error and discovery scales are insufficient.}
The separation $\eta\le\tau/2$ in the corollary cannot be replaced by
an assertion of small width at the threshold itself. To see this, let
$e_0,e_1,\ldots,e_n$ be orthonormal and $a_j=\eta(e_0+e_j)$. Every
$a_j$ is within $\eta$ of the line spanned by $e_0$ and has norm
$\sqrt2\eta$, but a direct least-squares calculation gives
\[
 \dist(a_j,\spanop(a_1,\ldots,a_{j-1}))^2=\eta^2(1+1/j)>\eta^2.
\]
A threshold-$\eta$ append rule therefore discovers every vector, although
all vectors are close to the same line. This refutes the inference that
width at the discovery threshold alone bounds generic irreversible span
discovery. It is not a lower bound for every convex optimizer.

\subsection{Optional dimension-free fallback}
A second query stream can provide a bound that needs no low-dimensional
structure. Interleave the wrapper with projected subgradient descent on
the original ball. For a horizon of $T$ fallback calls, start at $x_1=0$
and use the constant step size $\gamma=R/(G\sqrt T)$. The projection
inequality gives, for every $z\in X$,
\[
 \sum_{t=1}^T\ip{g_t}{x_t-z}
 \le\frac{\norm{x_1-z}^2}{2\gamma}
       +\frac\gamma2\sum_{t=1}^T\norm{g_t}^2
 \le RG\sqrt T.
\]
The observed protocol residual
$\max_{z\in X}T^{-1}\sum_t\ip{g_t}{x_t-z}$ is therefore at most
$RG/\sqrt T$. Jensen's inequality and the inexact oracle turn this into
a value-free gap certificate for the uniform average of at most
$\delta+RG/\sqrt T$. An exact streaming aggregate of the fallback slopes,
scalar products, and query points takes $O(d)$ memory; the full fallback
archive is unnecessary.

Interleaving the two independent streams changes their oracle counts by at
most a factor two. With the wrapper's tolerances chosen for the same target
$\delta+\eps$, return whichever stream first supplies its certificate.
The resulting call bound is the smaller of the intrinsic-dimension bound
and $O((RG/\eps)^2)$, without a compressibility assumption.

\section{Experimental details and additional results}
\label{app:channel-experiments}

\subsection{Data and objectives}
We use the UCI Abalone, Airfoil Self-Noise, Adult, and Breast Cancer
Wisconsin (Diagnostic) datasets
\citep{abalone_data,airfoil_data,adult_data,breast_cancer_data}.
The regression responses are ring count and sound-pressure level;
the classification labels are income category and diagnosis.
We randomly split each dataset into 70\% training, 15\% validation,
and 15\% test data, stratifying classification splits, with seed 20261007.
The training set is then capped at 2048 observations.
Adult uses the original training portion after removing observations
with missing values, with a new validation/test split from that portion.
Numeric predictors are standardized using training statistics, and
categorical predictors are one-hot encoded using training categories.
Regression responses are standardized using training statistics.

For preprocessed inputs $z_i$, define
\[
 K_{ij}=\exp\!\left(-\frac{\norm{z_i-z_j}^2}{2\sigma^2}\right),
 \qquad \sigma=\operatorname{median}_{i<j}\norm{z_i-z_j}.
\]
We use a full spectral factor $A$ satisfying $AA^\top=K$ and predictions
$Ax+c\ind$. All positive eigenvalues are retained, so $d=m$ for
$m$ training observations. The objectives are
\begin{align}
 F_{\rm SVM}(x)
 &=\frac{\lambda}{2}\norm{x}^2+
   \min_{c\in\R}\frac1m\sum_{i=1}^m
   \max\{0,1-y_i((Ax)_i+c)\},
 \label{eq:channel-svm}\\
 F_{\nu\text{-}\rm SVR}(x)
 &=\frac{\lambda}{2}\norm{x}^2+
   \min_{\substack{c\in\R\\\delta\ge0}}
   \left[\nu\delta+\frac1m\sum_{i=1}^m
   \max\{0,|y_i-(Ax)_i-c|-\delta\}\right],
 \qquad \nu=0.1.
 \label{eq:channel-svr}
\end{align}
The intercept and, for $\nu$-SVR, the tube width are minimized exactly
using order statistics at each query. The resulting objectives are
$\lambda$-strongly convex in $x$. We select $\lambda$ from
$\{10^{-4},10^{-3},10^{-2}\}$ by validation accuracy or RMSE using
specialized SVM/SVR solvers. All four datasets select $10^{-4}$.
Kernel width is fixed throughout.
These solvers are used for model selection and reference values only.

\begin{table}[t]
\centering\small
\caption{Selected settings. Step scales and threshold multipliers are
listed in SG/Mom/Ada order. The threshold is
$\tau=r\sqrt{\lambda\eps}$ with $\eps=10^{-3}F(0)$.
All tasks use $\lambda=10^{-4}$.}
\label{tab:channel-settings}
\setlength{\tabcolsep}{7pt}
\begin{tabular}{lcrcc}
\toprule
Data & Train/validation/test & $\sigma$ & Step scales & $r$\\
\midrule
Abalone & 2048/627/627 & 2.867 & $300/1000/0.3$ & $0.3/0.3/0.1$\\
Airfoil & 1052/225/226 & 2.852 & $333.3/1000/1$ & $0.3/0.3/0.3$\\
Adult & 2048/4524/4525 & 4.037 & $1000/1000/1$ & $0.3/0.3/0.3$\\
Breast Cancer & 398/85/86 & 6.334 & $1000/1000/3$ & $0.3/0.1/0.03$\\
\bottomrule
\end{tabular}
\end{table}

\subsection{Query methods and parameter selection}
All methods start from $x_0=0$. For update index $t\ge1$ and step
scale $s$, SG uses $x_t=x_{t-1}-s g_{t-1}/\sqrt t$.
Momentum uses $m_t=0.9m_{t-1}+0.1g_{t-1}$ and
$x_t=x_{t-1}-s m_t/\sqrt t$, with $m_0=0$.
Diagonal AdaGrad uses
$D_t=D_{t-1}+g_{t-1}^{\odot2}$ and
$x_t=x_{t-1}-s g_{t-1}/(\sqrt{D_t}+10^{-8})$, with coordinatewise
operations and $D_0=0$. AdaGrad is applied in the fixed spectral
coordinates of $A$ and therefore depends on this choice of coordinates.
Every 16 updates we also query the mean of the preceding block of
iterates; its subgradient is recorded for certification but does not
enter the method's updates. These additional queries count toward $N$.

Step scales are selected by the smallest incumbent objective after
500 queries. The coarse grid is
$\{0.1,1,10,100,1000\}$ for SG and momentum, and
$\{0.01,0.03,0.1,0.3,1\}$ for AdaGrad. We additionally try one third
and three times the best coarse scale.
Certificate gaps, stopping times, and test performance are not used
for this selection.

The compression threshold is selected separately from
$\{0.03,0.1,0.3\}\sqrt{\lambda\eps}$ by the smallest preliminary
time projection at 1 MiB/s among candidates reaching the target.
Table~\ref{tab:channel-settings} lists the chosen multipliers.
The append-only basis has capacity $\min(256,d)$.
If a slope's projection residual exceeds $\tau$ and capacity remains,
we append a reorthogonalized direction. At capacity we retain the
actual residual norm without expanding the basis.
Earlier records keep their original coordinates, residual radii,
and birth dimensions. Thus reaching capacity can weaken the bound,
but does not change its interpretation.

\subsection{Certificate construction and stopping}
At each query we form the loss slope $a_i=g_i-\lambda x_i$ and its
intercept $b_i=h(x_i)-\ip{a_i}{x_i}$ before compression.
Full maximizes \eqref{eq:baseline-quadratic} over $\Delta_N$, using a
CPU archive and an incrementally formed Gram matrix.
Compressed keeps the basis on the GPU, transmits
$(v_i,\eta_i,m_i)$, and solves
\eqref{eq:baseline-weight-master} using these records.
Both weight-selection problems contain all $N$ cuts.
Single-cut uses
\[
 L_{\rm sc}=\max\left\{0,
   \max_{i\le N}\left[F(x_i)-\frac{\norm{g_i}^2}{2\lambda}\right]
   \right\}.
\]
Each variant retains the known lower bound zero, the best single-cut
bound, and previous archive bounds. Masters are solved at the first
query and every 32 queries thereafter; the incumbent and single-cut bound
are updated at every query. Stopping can therefore occur between
master solves. We use Clarabel with feasibility and gap tolerances
$10^{-9}$, at most 300 iterations and 60 seconds per master.
Recovered weights are made simplex-feasible and evaluated in the
corresponding certificate formula.
All main-run masters report successful solution.
The common limits are 2000 queries and 600 seconds of computation,
excluding the emulated channel wait.

\subsection{Communication and timing}
We model a synchronous channel between a GPU worker and the
certification service. A common message contains three double-precision
scalars: $F(x_i)$, $b_i$, and $\norm{g_i}^2$.
Full additionally sends $d$ doubles, while compressed sends $m_i+2$
doubles for the coefficients, residual radius, and birth dimension.
The basis remains on the GPU; single-cut sends only the common message.
For a message of $S$ bytes, we insert a wall-clock wait
$S/B_{\rm link}$, with no fixed latency and no overlap with computation.
The primary bandwidth is $B_{\rm link}=1$ MiB/s.
This represents the bandwidth available to the certificate stream,
for example under a shared-link quota.
All computations actually run on one machine; the added wait emulates
the channel rather than implementing a distributed system.

We use an Intel Core i7-8550U and an NVIDIA GTX 1050 Max-Q, FP64
arithmetic, and one CPU thread. Timings include oracle evaluations,
updates, averaging, local GPU-to-host transfers, archive construction,
weight selection, monitoring, and the imposed channel delay.
Data preparation, kernel factorization, initialization, and test scoring
are excluded. The three variants run serially in a warmed process,
with cyclic ordering across three timing repetitions.
Their common query prefixes agree exactly in the recorded values.
These repetitions measure timing variability on fixed objectives.

\begin{table}[t]
\centering\small
\caption{Timing ranges across three main repetitions, final single-cut gaps
after 2000 queries, and test performance of the full-arm incumbent.
Times in Table~\ref{tab:channel-main} are medians of these repetitions.
Test performance is RMSE in standardized response units for regression
and accuracy for classification. Test data are not used for selection.}
\label{tab:channel-ranges}
\setlength{\tabcolsep}{6pt}
\begin{tabular}{llcccc}
\toprule
Data & Method & Full range (s) & Comp. range (s) & Single-cut gap & Test\\
\midrule
Abalone & SG & $[9.99,10.26]$ & $[3.63,5.07]$ & $6.78\!\times\!10^{-4}$ & 0.735\\
 & Mom & $[2.25,2.29]$ & $[0.77,1.04]$ & $4.12\!\times\!10^{-4}$ & 0.735\\
 & Ada & $[4.59,4.85]$ & $[1.73,1.79]$ & $8.18\!\times\!10^{-4}$ & 0.739\\
\midrule
Airfoil & SG & $[10.47,12.24]$ & $[4.37,5.57]$ & $1.66\!\times\!10^{-3}$ & 0.574\\
 & Mom & $[1.42,1.47]$ & $[0.65,0.89]$ & $1.27\!\times\!10^{-3}$ & 0.573\\
 & Ada & $[2.47,2.48]$ & $[1.19,1.34]$ & $2.21\!\times\!10^{-3}$ & 0.575\\
\midrule
Adult & SG & $[16.29,18.29]$ & $[13.26,15.77]$ & $7.77\!\times\!10^{-3}$ & 0.852\\
 & Mom & $[3.76,3.92]$ & $[1.46,1.74]$ & $6.22\!\times\!10^{-3}$ & 0.852\\
 & Ada & $[7.33,7.68]$ & $[3.53,4.32]$ & $6.51\!\times\!10^{-3}$ & 0.851\\
\midrule
Breast Cancer & SG & $[3.34,3.93]$ & $[2.15,2.70]$ & $1.08\!\times\!10^{-2}$ & 0.988\\
 & Mom & $[2.11,2.35]$ & $[1.58,1.75]$ & $9.38\!\times\!10^{-3}$ & 1.000\\
 & Ada & $[12.46,13.51]$ & $[8.65,9.67]$ & $1.57\!\times\!10^{-2}$ & 0.988\\
\bottomrule
\end{tabular}
\end{table}

\subsection{Bandwidth and accuracy sensitivity}
Figure~\ref{fig:channel-bandwidth} projects stopping time as
\[
 T(B_{\rm link})=T_{\rm native}+S/B_{\rm link},
\]
using separate preliminary runs without the imposed wait and their
total communicated bytes. The points at 0.1 and 10 MiB/s are
projections; the main 1 MiB/s results are measured with actual waiting.
Lower bandwidth increases the benefit of the smaller messages.
Without the imposed delay, compressed is faster in two of the twelve
preliminary comparisons; for the other cases, communication savings
offset its additional computation as bandwidth decreases.
On Abalone/SG the main measurements save 5.79 seconds
of channel wait, with other costs decreasing by 0.36 seconds;
on Adult/SG the 8.40 seconds saved are largely consumed by 7.05 seconds
of additional work. These are differences of component medians.

\begin{figure}[t]
\centering
\includegraphics[width=0.92\textwidth]{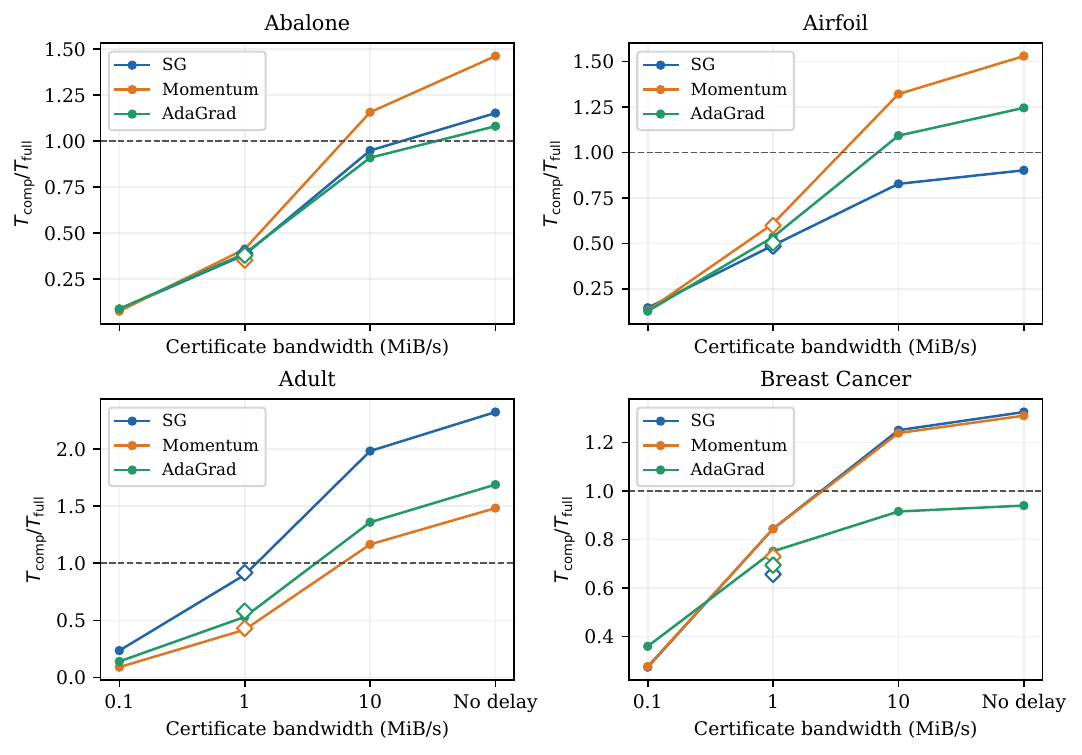}
\caption{Dependence on the available certificate bandwidth.
Lines use the preliminary timing projection; open diamonds show
ratios of median times from the three main measurements at 1 MiB/s.
The horizontal line denotes equal stopping times.
The last category is the preliminary run without an added channel delay.}
\label{fig:channel-bandwidth}
\end{figure}

Table~\ref{tab:channel-sensitivity} gives small sensitivity studies on
Abalone and Adult with SG, chosen before inspecting timing gains.
Thresholds are varied at the main accuracy using the preliminary runs;
accuracy is varied along separate continued trajectories with the selected
threshold fixed. Each of these comparisons uses one timing run.
At the main target all three tested thresholds retain a projected
advantage in the preliminary runs. The additional continued run on
Adult/SG projects a ratio of 1.007 at 1 MiB/s, consistent with the
small margin in this case.
At the coarse target $10^{-2}F(0)$ on Abalone, single-cut reaches the target
in 157 queries and projects to 1.00 seconds, compared with 1.12 seconds
for compressed and 2.18 seconds for full.

\begin{table}[t]
\centering\small
\caption{SG sensitivity. Threshold multiplier $r$ is varied at
$\eps/F(0)=10^{-3}$; accuracy is varied with the selected threshold
fixed. A dash means the target was not reached within 2000 queries.
Time ratios use the single-run projection at 1 MiB/s.}
\label{tab:channel-sensitivity}
\setlength{\tabcolsep}{8pt}
\begin{tabular}{llccrr}
\toprule
Data & Parameter & Value & Queries full/comp./single-cut & $k$ & $T_{\rm C}/T_{\rm F}$\\
\midrule
Abalone & $r$ & 0.03 & 384/384/--- & 133 & 0.437\\
 & $r$ & 0.1 & 384/384/--- & 116 & 0.414\\
 & $r$ & 0.3 & 384/385/--- & 83 & 0.382\\
 & $\eps/F(0)$ & $10^{-2}$ & 96/96/157 & 71 & 0.514\\
 & $\eps/F(0)$ & $10^{-3}$ & 384/385/--- & 83 & 0.695\\
 & $\eps/F(0)$ & $10^{-4}$ & 1652/1728/--- & 85 & 0.442\\
\midrule
Adult & $r$ & 0.03 & 612/731/--- & 256 & 0.937\\
 & $r$ & 0.1 & 612/731/--- & 256 & 0.940\\
 & $r$ & 0.3 & 612/714/--- & 256 & 0.899\\
 & $\eps/F(0)$ & $10^{-2}$ & 102/102/--- & 102 & 0.440\\
 & $\eps/F(0)$ & $10^{-3}$ & 612/714/--- & 256 & 1.007\\
 & $\eps/F(0)$ & $10^{-4}$ & ---/---/--- & --- & ---\\
\bottomrule
\end{tabular}
\end{table}

At $10^{-4}F(0)$ on Adult neither archive reaches the target.
The final full/compressed gaps are $1.99\times10^{-4}$ and
$2.63\times10^{-4}$, versus a target of $4.98\times10^{-5}$.
The common incumbent exceeds an independently computed feasible
reference value by $1.91\times10^{-4}$, already larger than the target;
the primal trajectory itself is therefore insufficiently accurate.
Compression also weakens the bound after the basis reaches capacity.
On Abalone both archives reach the stricter target, with the compressed
basis growing to only 85 directions. Here compression also reduces
computation without the imposed channel delay: the continued full and
compressed runs take 80.0 and 45.6 seconds, respectively.

Numerical checks of the profiled oracles, supporting inequalities,
simplex weights, and basis orthogonality show no inconsistency.
All reported gaps are evaluated in FP64.

\section{Details of the dimension-adaptation experiment}
\label{app:cp-experiment}

\subsection{Construction and oracle}
For each ambient dimension $d$, we orthonormalize a Gaussian matrix with
16 columns, using random seed 0. The first $k$ columns define $U$, so
different $k$ share a nested embedding at each $d$.
The objective \eqref{eq:cp-experiment-objective} is inspired by the
max-of-coordinates example in \citet[Section~3.2.1]{nesterov2018}.
We use ball radius $R=2$
and quadratic coefficient one. Its exact solution follows from
\[
 \sqrt{k}\max_j\ip{u_j}{x}
 \ge\ip{U\ind/\sqrt{k}}{x},\qquad
 \norm{U\ind/\sqrt{k}}=1.
\]
Completing the square gives $F(x)\ge-1/2$, with equality at
$x_*=-U\ind/\sqrt{k}$, which lies strictly inside the ball.
Thus both the initial gap and the solution norm are independent of $d,k$.

Each objective call returns only
\[
 g(x)=\sqrt{k}u_{j_*}+x,\qquad
 j_*:=\min\argmax_{1\le j\le k}\ip{u_j}{x}.
\]
This is an exact subgradient, with $\delta_i=0$ and
$\norm{g(x)}\le G:=\sqrt{k}+2$ on the ball.
The adaptive and ambient variants are not given $U$.
The known-subspace variant receives $U$ and maps its $k$ working
coordinates to the ambient oracle.
Function values and the explicit optimum are used only to check the
reported gaps; they do not affect queries or stopping.

\subsection{Common cutting-plane procedure and certificate}
We use the central-cut version of the volumetric method analyzed by
\citet{anstreicher1998}, with the explicit update rule in Algorithm~1 of
\citet{blanchard2023}.
Starting from $[-R,R]^q$, the procedure approximately minimizes
\[
 V(y)=\frac12\log\det\!\left(A^\top S(y)^{-2}A\right),\qquad
 S(y)=\operatorname{diag}(b-Ay),
\]
for the localizer $Ay<b$.
It inserts a central halfspace at the current center or deletes the
constraint with smallest leverage score if that score is below $0.04$.
Deletion is followed by recentering before another oracle call.
Centers are computed by damped Newton steps with decrement at most
$10^{-5}$; a QR factorization of $S(y)^{-1}A$ avoids forming ill-conditioned
normal equations. When barrier values are indistinguishable numerically,
the line search also checks a decrease in the Newton decrement.
This is a volumetric center, rather than the analytic center used in
other practical cutting-plane methods.

The adaptive variant starts with an empty basis and capacity $q=1$.
It appends a reorthogonalized gradient residual when its norm exceeds
$\tau=10^{-10}$. On overflow it doubles $q$, retains the discovered basis,
and discards the preceding phase's records, as in
Algorithm~\ref{alg:adaptive-cp}. The other variants use fixed dimensions
$q=k$ and $q=d$, respectively.
Records are retained only while their associated objective cuts remain
in the localizer. Ball-separation cuts are handled by the same procedure.

At each center, after synchronizing retained records, we minimize
$C(\alpha)$ from \eqref{eq:ball-certificate} over simplex weights using
its second-order-cone formulation and the historical norm recursion.
The conic solver uses feasibility and gap tolerances of $10^{-10}$.
The recovered weights are made feasible by clipping negative
entries and normalizing, and the certificate is evaluated directly from
these feasible weights. A numerical allowance of $10^{-11}$ is added.
We stop at the first check with $C\le\eps$, returning the corresponding
weighted average. With $\zeta=\eps-R\tau$, this is the stopping threshold
$\delta+R\tau+\zeta$ in Theorem~\ref{thm:cp}.

We count every objective-oracle call, including overflow queries and
queries repeated after a restart. Centering, constraint deletion, weight
selection, and ball separation require no objective-oracle call.
Several accuracy targets are read from the first certificate crossings
of one trajectory, since the target affects only stopping.
The common limits of 12,000 objective calls and 180 seconds per run are
never reached. The 31 runs give 43 successful target crossings; the
additional four targets for the ambient variant at $d=64,k=4$ also
include $\eps=10^{-6}$, reached in 694 calls.

\subsection{Complete dimension and accuracy counts}
Tables~\ref{tab:cp-dimension-details} and
\ref{tab:cp-accuracy-details} give the counts underlying
Figure~\ref{fig:adaptive-cp-experiment}.
The known-subspace counts in the dimension sweep are identical across
the three ambient dimensions. Every adaptive run discovers exactly $k$
directions; its final capacity is $q=k$ for the tested powers of two.
Small differences between embeddings occur in the floating-point
selection of active indices and centering on this symmetric objective.

\begin{table}[htbp]
 \centering
 \caption{Objective-oracle calls at $\eps=10^{-4}$.
 The first three rows use the adaptive method; the last row gives the
 common known-subspace counts.}
 \label{tab:cp-dimension-details}
 \begin{tabular}{lrrrr}
  \toprule
  & $k=2$ & $k=4$ & $k=8$ & $k=16$\\
  \midrule
  Adaptive, $d=64$   & 23 & 52 & 104 & 207\\
  Adaptive, $d=256$  & 23 & 52 & 105 & 208\\
  Adaptive, $d=1024$ & 23 & 52 & 105 & 207\\
  Known subspace    & 21 & 47 & 95  & 189\\
  \bottomrule
 \end{tabular}
\end{table}

\begin{table}[htbp]
 \centering
 \caption{Objective-oracle calls at $d=256,k=8$ for each certified
 accuracy.}
 \label{tab:cp-accuracy-details}
 \begin{tabular}{rrr}
  \toprule
  $\eps$ & Known subspace & Adaptive\\
  \midrule
  $10^{-2}$ & 54  & 64\\
  $10^{-3}$ & 75  & 85\\
  $10^{-4}$ & 95  & 105\\
  $10^{-5}$ & 116 & 126\\
  $10^{-6}$ & 137 & 147\\
  \bottomrule
 \end{tabular}
\end{table}

\subsection{Numerical checks and interpretation}
We check subgradient support inequalities, barrier derivatives against
finite differences, and the volumetric center against an independent
optimization solve. An independent weight optimization checks the
historical certificate on small examples with nonzero residual profiles.
For all 31 optimization runs, we reconstruct the final output from the
retained points and weights and independently reevaluate its certificate.
All recorded checks satisfy $F(\bar x)-F_*\le C$, and all stopping checks
satisfy $C\le\eps$. The numerical dual gap at each reported endpoint is
below $\eps/4$.
The largest basis orthogonality error is $1.2\times10^{-15}$, and the
largest omitted gradient norm is $6.4\times10^{-15}$.
At most $4q$ constraints are retained during these runs, below the
$25q+1$ bound for the central-cut primitive.

The controlled construction isolates the discovered-dimension dependence
and permits direct checks of the value-free stopping certificate.
These are floating-point illustrations of the theoretical contribution,
not interval-certified executions of the exact-arithmetic theorem.
The experiment uses one embedding per ambient dimension and effectively
exact-rank gradients; it does not test approximate-rank trajectories or
the benefit of historical residual information at a nonnegligible
compression scale.

\end{document}